\documentclass[11pt, notitlepage]{article}
\usepackage{amssymb,amsmath,comment}
\usepackage{amsthm}
\catcode`\@=11 \@addtoreset{equation}{section}
\def\thesection{\arabic{section}}

\def\theequation{\thesection.\arabic{equation}}
\catcode`\@=12
\usepackage{colortbl}
\usepackage{hyperref}
\usepackage[mathscr]{eucal}
\usepackage{epsf}
\usepackage{esint}
\usepackage{a4wide}

\newcommand{\Om} {\Omega}

\newcommand{\noi} {\noindent}
\newcommand{\na} {\nabla}

\usepackage[all]{xy}
\catcode`\@=11
\def\theequation{\@arabic{\c@section}.\@arabic{\c@equation}}
\catcode`\@=12

\newtheorem{Theorem}{Theorem}[section]
\newtheorem{Lemma}[Theorem]{Lemma}

\newtheorem{Remark}[Theorem]{Remark}
\newtheorem{Definition}[Theorem]{Definition}

\begin{document}

{\vspace{0.01in}}

\title{Mixed local-nonlocal $p$-Laplace equation with variable singular nonlinearity in the Heisenberg group}

\author{Prashanta Garain\footnote{Department of Mathematical Sciences\\
Indian Institute of Science Education and Research Berhampur\\ Permanent Campus, At/Po:-Laudigam, Dist.-Ganjam\\
Odisha, India-760003\\
\textsf{Email}: pgarain92@gmail.com}
}

\maketitle

\begin{abstract}\noindent
We investigate a mixed local-nonlocal $p$-Laplace equation on the Heisenberg group, where the nonlinear term features a variable singular exponent. Our analysis establishes the existence, uniqueness, and regularity of weak solutions under suitable structural assumptions. To the best of our knowledge, this work provides the first treatment of such mixed local-nonlocal problems in a non-commutative setting, even in the linear case $p=2$ with a constant singular exponent.
\end{abstract}

\maketitle

\noi {Keywords: Mixed local-nonlocal $p$-Laplace equation, variable singular exponent, existence, uniqueness, regularity, Heisenberg group.}

\noi{\textit{2020 Mathematics Subject Classification: 35J62, 35J75, 35R03, 35R11, 35A02.}

\bigskip

\tableofcontents

\section{Introduction and main results}
\subsection{Introduction}
In this article, we investigate a class of mixed local-nonlocal $p$-Laplace equations with variable singular exponent posed on the Heisenberg group $\mathbb{H}^N$. More precisely, we consider the problem
\begin{equation}\label{meqn}
M_{\alpha}\,u=f(x)\,u^{-\delta(x)} \quad \text{in }\Omega,\qquad
u>0 \text{ in }\Omega,\qquad
u=0 \text{ in }\mathbb{H}^N\setminus\Omega,
\end{equation}
where $\Omega\subset \mathbb{H}^N$ is a bounded smooth domain, $0<s<1<p<Q$, with $Q=2N+2$ being the homogeneous dimension of $\mathbb{H}^N$ and
\[
M_\alpha u = -\Delta_{H,p} u + \alpha (-\Delta_{H,p})^s u
\]
denotes a mixed operator combining the local and fractional $p$-subLaplace components. Here, $\alpha> 0$, and 
\[
\Delta_{H,p} u = \operatorname{div}(|\nabla_H u|^{p-2}\nabla_H u)
\]
is the $p$-subLaplacian, while the fractional counterpart is defined by
\[
(-\Delta_{H,p})^s u(x) = \mathrm{P.V.}\int_{\mathbb{H}^N}
\frac{|u(x)-u(y)|^{p-2}(u(x)-u(y))}{|y^{-1}\circ x|^{Q+sp}}\,dy,
\]
with $\mathrm{P.V.}$ indicating the principal value. The data satisfy $f\in L^m(\Omega)\setminus\{0\}$ for some $m\ge 1$, and the singular exponent $\delta:\overline{\Omega}\to(0,\infty)$ is continuous (see the main results for precise hypotheses). The singularity in \eqref{meqn} stems from the blow-up behavior of $u^{-\delta(x)}$ as $u\to 0^{+}$.

Singular elliptic equations have been an active area of research for more than three decades, particularly within the Euclidean setting, where both purely local and purely nonlocal models have been extensively explored. In the local Euclidean case, corresponding to $-\Delta_p=\text{div}(|\na u|^{p-2}\na u)$ in place of $M_\alpha$, equation~\eqref{meqn} and its variants has been studied for a wide range of assumptions on the datum $f$ and for both constant and spatially varying singular exponents $\delta$. For constant $\delta>0$ in the semilinear case $p=2$, we refer to \cite{Boc-Or, CRT, Lz}; for the quasilinear range $1<p<\infty$, we point to \cite{Canino, DeCave, Garainmn} and the references therein. When $\delta$ is allowed to vary spatially, the first contribution is due to \cite{CMP} for $p=2$, followed by several important developments such as \cite{Alves, BGM, Garainmm}.

In recent years, the nonlocal regime---that is, the problem
\begin{equation}\label{meqnoc}
\begin{split}
(-\Delta_p)^s u&:= \mathrm{P.V.}\int_{\mathbb{H}^N}
\frac{|u(x)-u(y)|^{p-2}(u(x)-u(y))}{|y^{-1}\circ x|^{Q+sp}}\,dy
=f(x)\,u^{-\delta(x)}\quad \text{in }\Omega,\\
u&>0 \text{ in }\Omega,\qquad u=0 \text{ in }\mathbb{R}^N\setminus\Omega,
\end{split}
\end{equation}
and its variants has likewise received considerable attention. For constant singular exponents $\delta>0$ and $p=2$, we refer to \cite{Nsop, Ncvpde}, while for $1<p<\infty$, the works \cite{Caninoetal, Nmn, Mukanona1, Mukanona2} offer a broad overview. Contributions dealing with variable singular exponents in the nonlocal context may be found in \cite{Garaincpaa} and the references therein.

A natural intermediate setting arises when both local and nonlocal effects are present. In the mixed local–nonlocal regime, for $M_{1}$ replaces with
\[
-\Delta_p + (-\Delta_p)^s,
\]
singular problems have been recently addressed in the Euclidean framework. For constant $\delta>0$ in the semilinear case $p=2$, we refer to \cite{AGmz, ARad, Vecchi, Garainjga}; for the quasilinear case $1<p<\infty$, see \cite{BalDas1, Dhanya, GU21} and the references therein. Results for variable singular exponents $\delta(x)$ in this mixed setting appear in \cite{Biroud, GKK}.

Despite this substantial progress, \emph{mixed local–nonlocal singular problems in non-Euclidean geometries remain completely unexplored}. To the best of our knowledge, the present work constitutes the first attempt to study such equations in the non-Euclidean geometry of the Heisenberg group. Before presenting our main results, let us briefly outline the existing literature in the purely local and purely nonlocal cases under non-Euclidean settings.

For the local singular problem
\[
-\Delta_p u = f(x) u^{-\delta} \quad \text{in }\Omega,\qquad 
u>0 \ \text{in }\Omega,\qquad 
u=0 \ \text{on }\partial\Omega,
\]
posed in a bounded domain $\Omega$ of a Carnot group $G$, existence results for constant $\delta\in(0,1)$ and nonnegative $f\in L^1(\Omega)$ were obtained in \cite{GUaamp}. We also refer to \cite{GK, KD, Pucciacv, Sahu, WW} for related developments. In the purely nonlocal setting, we mention \cite{GKR}, where existence was established for constant $\delta\in(0,1)$ and multiplicity results were obtained for a perturbed nonlocal equation on a stratified Lie group. Very recent progress in the Heisenberg group, including both constant and variable singular exponents, can be found in \cite{GarainHein}.

For the mixed local–nonlocal operator, the only available results in non-Euclidean geometries are \cite{DC, Zhang, Zhang2}, where existence and regularity was established for non-singular equations. Importantly, the regularity estimate obtained in \cite{Zhang} forms a cornerstone of our approach to the singular problem \eqref{meqn}. Our strategy for proving existence follows the classical approximation method of \cite{Boc-Or}, suitably adapted to handle variable exponents as in \cite{CMP} and the mixed structure as in \cite{GU21}. We construct approximate solutions and show that they remain uniformly positive on compact subsets $\omega\Subset\Omega$, thanks to a weak Harnack inequality inspired by \cite{Zhang}. To pass to the limit in the nonlinear terms, we develop a gradient convergence result (Theorem~\ref{thm:A1}). Uniqueness follows from an appropriate comparison principle, in the spirit of \cite{Caninouni}. Finally, we establish regularity results that depend on the behavior of the singular exponent $\delta$, based on suitable \emph{a priori} estimates for the approximate solutions.

\subsection{Notations and standing assumptions}

We collect below the main notations and assumptions that will be used throughout the paper unless otherwise mentioned.

\begin{itemize}
    \item For $l > 1$, we denote by $l' = \frac{l}{l-1}$ the conjugate exponent of $l$.

    \item We assume $0 < s < 1$ and $1 < p <  Q$, where $Q = 2N + 2$ denotes the homogeneous dimension of $\mathbb{H}^N$.

    \item The critical Sobolev exponent is given by
    \[
    p^* = \frac{Qp}{Q - p}, \qquad 1 < p < Q.
    \]

    \item For $q > 1$, we define $J_q(t) = |t|^{q-2}t$ for all $t \in \mathbb{R}$.

    \item We use the notation 
    \[
    d\mu = \frac{dx\, dy}{|y^{-1} \circ x|^{Q + sp}}.
    \]

    \item The symbol $\Omega$ denotes a bounded smooth domain in $\mathbb{H}^N$.

    \item By $\omega \Subset \Omega$, we mean that $\omega$ is compactly contained in $\Omega$, i.e.\ $\omega \subset \overline{\omega} \subset \Omega$.

    \item For $u \in HW_0^{1,p}(\Omega)$, the norm of $u$ (defined below in \eqref{eqnorm}) is written as
    \[
        \|u\| := \|u\|_{HW_0^{1,p}(\Omega)}
        = \left( \int_{\Omega} |\nabla_H u|^p \, dx 
        + \alpha\int_{\mathbb{H}^N} \int_{\mathbb{H}^N} 
        \frac{|u(x) - u(y)|^p}{|y^{-1} \circ x|^{Q + sp}} \, dx \, dy \right)^{1/p}.
    \]
   \item  If $u \in L^{p}(\Omega)$ for some $1 \le p \le \infty$, we denote its $L^p$-norm by $\|u\|_{L^{p}(\Omega)}$.

    \item For $k \in \mathbb{R}$, we use the standard notation 
    \[
    k^{+} = \max\{k,0\}, \qquad 
    k^{-} = \max\{-k,0\}, \qquad 
    k_{-} = \min\{k,0\}.
    \]

    \item For a measurable function $F$ on a set $S$ and constants $c,d$, the notation $c \le F \le d$ in $S$ means that $c \le F(x) \le d$ for almost every $x \in S$.

    \item The symbol $C$ denotes a generic positive constant whose value may change from line to line. If the constant depends on certain parameters $r_1, r_2, \ldots, r_k$, we write $C = C(r_1, r_2, \ldots, r_k)$.
\end{itemize}

Before stating our main results, we define the condition $(P_{\epsilon,\delta_*})$ and we fix a given $\alpha\geq 0$ for the problem \eqref{meqn}.

\textbf{Condition $(P_{\epsilon,\delta_*})$:} 
A continuous function $\delta : \overline{\Omega} \to (0,\infty)$ is said to satisfy $(P_{\epsilon,\delta_*})$ if there exist constants $\delta_* \geq 1$ and $\epsilon > 0$ such that 
\[
\delta(x) \leq \delta_* \quad \text{for all } x \in \Omega_\epsilon,
\]
where
\[
\Omega_\epsilon := \{\, y \in \Omega : \mathrm{dist}(y, \partial\Omega) < \epsilon \,\}.
\]
In addition, we set
\begin{equation}\label{deltanbd2}
\omega_\epsilon := \Omega \setminus \overline{\Omega_\epsilon}, \qquad \epsilon > 0.
\end{equation}

\subsection{Main results}

\textbf{Existence:} Our first existence result reads as follows: 
\begin{Theorem}\label{varthm1}(\textbf{Variable singular exponent})
Let $\delta : \overline{\Omega} \to (0,\infty)$ be a continuous function satisfying the condition $(P_{\epsilon,\delta_*})$ for some $\epsilon>0$ and $\delta_* \geq 1$. Then the problem \eqref{meqn} admits a weak solution 
\[
u \in HW_{\mathrm{loc}}^{1,p}(\Omega) \cap L^{p-1}(\Omega)
\]
such that 
\[
u^{\frac{\delta_{*}+p-1}{p}} \in HW_0^{1,p}(\Omega),
\]
provided that $f \in L^{m}(\Omega)\setminus\{0\}$ is nonnegative, where 
\[
m = \Bigg( \frac{(\delta_* + p - 1)p^{*}}{p \delta_*} \Bigg)' .
\]
\end{Theorem}

Our next existence result concerns the case of a constant singular exponent.

\begin{Theorem}\label{thm1}(\textbf{Constant singular exponent})
Suppose that $\delta : \overline{\Omega} \to (0,\infty)$ is a constant function, and let  
$f \in L^{m}(\Omega)\setminus\{0\}$ be nonnegative for some $m$. Then the problem \eqref{meqn} admits a weak solution $u$ in each of the following cases:
\begin{enumerate}
    \item[(a)] If $0 < \delta < 1$ and 
    \[
    m = \left( \frac{p^{*}}{1-\delta} \right)',
    \]
    then $u \in HW_0^{1,p}(\Omega)$.
    
    \item[(b)] If $\delta = 1$ and $m = 1$, then $u \in HW_0^{1,p}(\Omega)$.
    
    \item[(c)] If $\delta > 1$ and $m = 1$, then 
    \[
    u \in HW_{\mathrm{loc}}^{1,p}(\Omega) \cap L^{p-1}(\Omega)
    \quad \text{and} \quad
    u^{\frac{\delta + p - 1}{p}} \in HW_0^{1,p}(\Omega).
    \]
\end{enumerate}
\end{Theorem}

\textbf{Regularity:} Our regularity results are stated as follows:
\begin{Theorem}\label{regthm}(\textbf{Variable singular exponent})
Let $\delta : \overline{\Omega} \to (0,\infty)$ be a continuous function satisfying the condition $(P_{\epsilon,\delta_*})$ for some $\delta_* \geq 1$ and $\epsilon > 0$. 
Let $u$ be the weak solution to problem \eqref{meqn} provided by Theorem~\ref{varthm1}, and assume that 
$f \in L^{m}(\Omega)\setminus\{0\}$ is nonnegative for some $m$. Then:
\begin{enumerate}
    \item[(i)] If 
    \[
    m \in \left[ \frac{Q(\delta_* + p - 1)}{Q(p-1) + \delta_* p}, \, \frac{Q}{p} \right),
    \]
    then $u \in L^{t}(\Omega)$, where $t = m' \gamma$ and 
    \[
    \gamma = \frac{N(p-1)(m-1)}{Q-mp}.
    \]
    \item[(ii)] If $m > \frac{Q}{p}$, then $u \in L^{\infty}(\Omega)$.
\end{enumerate}
\end{Theorem}

\begin{Theorem}\label{regthm1}(\textbf{Constant singular exponent, case $0 < \delta < 1$})
Assume that $\delta : \overline{\Omega} \to (0,1)$ is constant. 
Let $u$ be the weak solution of \eqref{meqn} given by Theorem~\ref{thm1}-\textnormal{(a)}, and suppose 
$f \in L^{m}(\Omega)\setminus\{0\}$ is nonnegative for some $m$. Then:
\begin{enumerate}
    \item[(i)] If 
    \[
    m \in \left( \left( \frac{p^{*}}{1-\delta} \right)', \, \frac{Q}{p} \right),
    \]
    then $u \in L^{t}(\Omega)$, where $t = p^{*}\gamma$ and 
    \[
    \gamma = \frac{(\delta + p - 1)\gamma'}{p m' - p^{*}}.
    \]
    \item[(ii)] If $m > \frac{Q}{p}$, then $u \in L^{\infty}(\Omega)$.
\end{enumerate}
\end{Theorem}

\begin{Theorem}\label{regthm2}(\textbf{Constant singular exponent, case $\delta = 1$})
Assume $\delta \equiv 1$ in $\overline{\Omega}$. 
Let $u$ be the weak solution of \eqref{meqn} obtained in Theorem~\ref{thm1}-\textnormal{(b)}, and let 
$f \in L^{m}(\Omega)\setminus\{0\}$ be nonnegative for some $m$. Then:
\begin{enumerate}
    \item[(i)] If 
    \[
    m \in (1, \, Q/p),
    \]
    then $u \in L^{t}(\Omega)$, where $t = p^{*} \gamma$ with 
    \[
    \gamma = \frac{p m'}{p m' - p^{*}}.
    \]
    \item[(ii)] If $m > \frac{Q}{p}$, then $u \in L^{\infty}(\Omega)$.
\end{enumerate}
\end{Theorem}

\begin{Theorem}\label{regthm3}(\textbf{Constant singular exponent, case $\delta > 1$})
Let $\delta : \overline{\Omega} \to (1,\infty)$ be constant. 
Assume that $u$ is the weak solution of \eqref{meqn} given by Theorem~\ref{thm1}-\textnormal{(c)}, and that 
$f \in L^{m}(\Omega)\setminus\{0\}$ is nonnegative for some $m$. Then:
\begin{enumerate}
    \item[(i)] If 
    \[
    m \in (1, \, Q/p),
    \]
    then $u \in L^{t}(\Omega)$, where $t = p^{*}\gamma$ and 
    \[
    \gamma = \frac{(\delta + p - 1)m'}{p m' - p^{*}}.
    \]
    \item[(ii)] If $m > \frac{Q}{p}$, then $u \in L^{\infty}(\Omega)$.
\end{enumerate}
\end{Theorem}

\textbf{Uniqueness:} Our final result is the following uniqueness theorem.
\begin{Theorem}\label{thm4}(\textbf{Uniqueness})
Let $\delta : \overline{\Omega} \to (0,\infty)$ be a constant function, and suppose that 
$f \in L^{1}(\Omega)\setminus\{0\}$ is nonnegative. Then the problem \eqref{meqn} admits at most one weak solution in 
\[
HW^{1,p}_{\mathrm{loc}}(\Omega) \cap L^{p-1}(\Omega).
\]
\end{Theorem}

\begin{Remark}\label{rmk}
We note that by proceeding along the lies of the proof, all the main results presented above remain valid in the purely local setting, corresponding to $\alpha = 0$. To the best of our knowledge, these conclusions are also new in the context of variable singular exponents on the Heisenberg group.
\end{Remark}

\textbf{Organization of the paper:} The remainder of this article is structured as follows:
In Section 2, we introduce the functional framework and recall some auxiliary results that will be used throughout the paper.
In Section 3, we establish some preliminary results required to prove our main results.
Finally, in Section 4, we complete the proof of our main results.

\section{Functional setting and auxiliary results}
In this section, we briefly recall some basic facts about the Heisenberg group $\mathbb{H}^N$, introduce several function spaces.

The Euclidean space $\mathbb{R}^{2N+1}$, endowed with the group multiplication
\[
\xi \circ \eta =
\left(
x_1+y_1,\,
x_2+y_2,\,
\ldots,\,
x_{2N}+y_{2N},\,
\tau + \tau' + \frac{1}{2}\sum_{i=1}^{N}\big(x_i y_{N+i} - x_{N+i} y_i\big)
\right),
\]
where $\xi=(x_1,\ldots,x_{2N},\tau)$ and $\eta=(y_1,\ldots,y_{2N},\tau')\in\mathbb{R}^{2N+1}$, defines the Heisenberg group $\mathbb{H}^N$.

The left-invariant vector fields on $\mathbb{H}^N$ are given by
\[
X_i = \partial_{x_i} - \frac{x_{N+i}}{2}\,\partial_\tau,
\qquad
X_{N+i} = \partial_{x_{N+i}} + \frac{x_i}{2}\,\partial_\tau,
\qquad 1\le i\le N,
\]
and the nontrivial commutator is
\[
T = \partial_\tau = [X_i,\, X_{N+i}]
= X_i X_{N+i} - X_{N+i} X_i,
\qquad 1\le i\le N.
\]
We refer to $X_1, X_2, \ldots, X_{2N}$ as the horizontal vector fields on $\mathbb{H}^N$, and $T$ as the vertical vector field.

The Haar measure on $\mathbb{H}^N$ is equivalent to the Lebesgue measure on $\mathbb{R}^{2N+1}$.  
For a measurable set $E\subset \mathbb{H}^N$, we denote its Lebesgue measure by $|E|$.

For $\xi=(x_1,\ldots,x_{2N},\tau)$, we define its Kor\'anyi-type norm by  
\[
|\xi|
= \left( \left(\sum_{i=1}^{2N} x_i^2 \right)^{2} + \tau^{2} \right)^{1/4}.
\]
The Carnot-Carath\'eodory distance between two points $\xi,\eta\in\mathbb{H}^N$ is defined as the infimum of the lengths of horizontal curves joining them, and is denoted by $d(\xi,\eta)$.  
This distance is equivalent to the Kor\'anyi metric, i.e.,
\[
d(\xi,\eta)\sim |\xi^{-1}\circ\eta|.
\]

The ball of radius $r>0$ centered at $\xi_0$ with respect to $d$ is given by  
\[
B_r(\xi_0)=\{\xi\in\mathbb{H}^N:\ d(\xi,\xi_0)<r\}.
\]
When the center is irrelevant or understood from the context, we write simply $B_r := B_r(\xi_0)$.

The homogeneous dimension of $\mathbb{H}^N$ is $Q = 2N + 2$. Let $1 \le p < \infty$ and $\Omega \subset \mathbb{H}^N$.  
The Sobolev space $HW^{1,p}(\Omega)$ is defined by
\[
HW^{1,p}(\Omega)
=
\left\{
u \in L^{p}(\Omega) :
\nabla_H u \in L^{p}(\Omega)
\right\},
\]
where $\na_H u$ is the horizontal gradient of $u$ defined by  
\[
\na_H u = (X_1u, X_2u, \ldots, X_{2N}u).
\]
and is endowed with the norm
\[
\|u\|_{HW^{1,p}(\Omega)}
=
\|u\|_{L^{p}(\Omega)}
+
\left\|\nabla_H u\right\|_{L^{p}(\Omega)}.
\]
With this norm, $HW^{1,p}(\Omega)$ is a Banach space.

The local Sobolev space $HW_{\mathrm{loc}}^{1,p}(\Omega)$ is defined as
\[
HW_{\mathrm{loc}}^{1,p}(\Omega)
=
\left\{
u : u \in HW^{1,p}(\Omega') \ \text{for every open set } \Omega' \Subset \Omega
\right\}.
\]

Let $1 \le p < \infty$, $s \in (0,1)$, and let $v:\mathbb{H}^N \to \mathbb{R}$ be a measurable function.  
The Gagliardo seminorm of $v$ is defined by
\[
[v]_{HW^{s,p}(\mathbb{H}^n)}
=
\left(
\int_{\mathbb{H}^N}\int_{\mathbb{H}^N}
{|v(x)-v(y)|^{p}}
\,d\mu
\right)^{1/p}.
\]

The fractional Sobolev space on the Heisenberg group is then defined as
\[
HW^{s,p}(\mathbb{H}^N)
=
\left\{
v \in L^{p}(\mathbb{H}^N) : [v]_{HW^{s,p}(\mathbb{H}^N)} < \infty
\right\}.
\]
The space $HW^{s,p}(\mathbb{H}^N)$ is endowed with the natural fractional norm
\[
\|v\|_{HW^{s,p}(\mathbb{H}^N)}
=
\left( 
\|v\|_{L^{p}(\mathbb{H}^N)}^{p}
+
[v]_{HW^{s,p}(\mathbb{H}^N)}^{p}
\right)^{1/p}.
\]

For any open set $\Omega \subset \mathbb{H}^N$, the fractional Sobolev space $HW^{s,p}(\Omega)$ and its associated norm  
$\|v\|_{HW^{s,p}(\Omega)}$ are defined in an analogous manner.

To deal with the mixed problem \eqref{meqn} (that is when $\alpha>0$ in \eqref{meqn}), we consider the space $HW_{0}^{1,p}(\Omega)$ defined by
$$
HW_0^{1,p}(\Om)=\{u\in HW^{1,p}(\mathbb{H}^N):u=0\text{ in }\mathbb{H}^N\setminus\Om\}.
$$

The next result follows similarly as in \cite[Proposition 2.2]{Hitchhiker'sguide}.
\begin{Lemma}\label{l2.2}
Let $\Omega \subset \mathbb{H}^{N}$ be a smooth bounded domain, and assume that  
$1 < p < \infty$ and $0 < s < 1$. Then there exists a constant $C = C(Q,p,s) > 0$ such that
\[
\|u\|_{HW^{s,p}(\Omega)} \le C\, \|u\|_{HW^{1,p}(\Omega)}
\qquad \text{for all } u \in HW^{1,p}(\Omega).
\]
\end{Lemma}
The next result follows from \cite[Theorem 2.4]{DC}.
\begin{Lemma}\label{l2.3}
Let $\Omega$ be a bounded domain in $\mathbb{H}^N$, and assume $1 < p < \infty$ and $0 < s < 1$.  
Then there exists a constant $C = C(Q,p,s,\Omega) > 0$ such that
\[
\iint_{\mathbb{H}^N \times \mathbb{H}^N}
{|u(x)-u(y)|^{p}}\,d\mu
\le
C \int_{\Omega} |\nabla_H\,u(x)|^{p}\, dx,
\]
for every $u \in HW^{1,p}_0(\Omega)$.
\end{Lemma}

For the following embedding result, we refer to \cite[Theorem 8.1]{Koskela}.
\begin{Lemma}\label{emb}
Let $0<s<1$ and $1<p<Q$. Then the embedding
$$
HW^{1,p}(\Om)\hookrightarrow L^r(\Om)
$$
is continuous for every $r\in[1,p^*]$ and compact for every $r\in[1,p^*)$.
\end{Lemma}

Taking into account Lemma \ref{l2.2} and Lemma \ref{emb}, for $\alpha>0$, we consider the following equivalent norm $\|\cdot\|$ on $HW_0^{1,p}(\Om)$ defined by
\begin{equation}\label{eqnorm}
\|u\|_{HW_0^{1,p}(\Om)}:=\left(\int_{\Om}|\na_H u|^p\,dx+\alpha\int_{\mathbb{H}^N}\int_{\mathbb{H}^N}{|u(x)-u(y)|^p}\,d\mu\right)^\frac{1}{p}.
\end{equation}

To provide a rigorous framework for our analysis, we provide below the notion of weak solutions. For this, first we define the associated zero boundary condition as follows:

\begin{Definition}[Dirichlet Boundary Condition]\label{rediri}
Let $u$ be such that $u = 0$ in $\mathbb{H}^N \setminus \Omega$. We say that $u \leq 0$ on $\partial\Omega$ if, for every $\theta > 0$, it holds that $(u-\theta)^+\in HW_0^{1,p}(\Om)$. Further, we say that $u = 0$ on $\partial\Omega$ if $u$ is nonnegative and satisfies $u \leq 0$ on $\partial\Omega$.
\end{Definition}

We now define weak solutions to the problem \eqref{meqn}.

\begin{Definition}[Weak Solution]\label{wksoldef}
Let $f \in L^1(\Omega)\setminus{0}$ be nonnegative and let $\Omega\subset \mathbb{H}^N$ be a bounded smooth domain. A function $u \in HW^{1,p}_{\mathrm{loc}}(\Omega) \cap L^{p-1}(\Omega)$ is called a weak solution to problem \eqref{meqn} if $u=0$ in $\mathbb{H}^N\setminus\Om$ such that $u=0$ on $\partial\Om$ as in Definition \ref{rediri} and the following conditions hold:
\begin{itemize}
\item For every subset $\omega \Subset \Omega$, there exists a constant $C = C(\omega) > 0$ such that $u \geq C$ in $\omega$;
\item For every $\varphi \in C_c^1(\Omega)$, one has
\begin{equation}\label{wksoleqn}
\begin{split}
&\int_{\Om}|\na_H u|^{p-2}\na_H u\na_H\varphi\,dx+\alpha\int_{\mathbb{H}^N}\int_{\mathbb{H}^N}J_p(u(x)-u(y))(\varphi(x)-\varphi(y))\,d\mu\\
&\qquad= \int_{\Omega} f(x) u^{-\delta(x)} \varphi(x)\,dx.
\end{split}
\end{equation}
\end{itemize}
\end{Definition}

We observe that Definition \ref{wksoldef} is well stated by Lemma \ref{l2.2} and Lemma \ref{l2.3}.

\subsection*{Auxiliary results}
In this subsection, we present several auxiliary results. The first, taken from \cite[Theorem 9.14]{var}, plays a key role in establishing the existence of approximate solutions.
\begin{Theorem}\label{MB}
Let $V$ be a real separable reflexive Banach space and $V^*$ be the dual of $V$. Suppose that $T:V\to V^{*}$ is a coercive and demicontinuous monotone operator. Then $T$ is surjective, i.e., given any $f\in V^{*}$, there exists $u\in V$ such that $T(u)=f$. If $T$ is strictly monotone, then $T$ is also injective.  
\end{Theorem}

The following result holds along with the lines of the proof of \cite{Canino}.
\begin{Lemma}\label{mainappos}
Let $\gamma>0$ and let $u$ be nonnegative with $u^{\max\big\{\frac{\gamma+p-1}{p},1\big\}}\in HW^{1,p}_0(\Omega).$
Then $u$ fulfills zero  Dirichlet boundary conditions in the sense of Definition \ref{rediri}.
\end{Lemma}

For the following algebraic inequality, see \cite[Lemma 2.1]{Dama}.
\begin{Lemma}\label{alg}
Let $1<p<\infty$. Then for every $x,y\in\mathbb{R}^k$, there exists a constant $C=C(p)>0$ such that
$$
\langle J_p(x)-J_p(y),x-y\rangle\geq C(|x|+|y|)^{p-2}|x-y|^2.
$$
\end{Lemma}
Further, for the following algebraic inequality, we refer to \cite[Lemma A.2]{BP}.
\begin{Lemma}\label{BPalg}
Let $1<p<\infty$ and $g:\mathbb{R}\to\mathbb{R}$ be an increasing function. We define
$$
G(t)=\int_{0}^{t}g'(s)^\frac{1}{p}\,ds,\quad t\in\mathbb{R}.
$$
Then for every $a,b\in\mathbb{R}$, we have
$$
J_p(a-b)(g(a)-g(b))\geq |G(a)-G(b)|^p.
$$
\end{Lemma}

\section{Preliminaries}
\subsection{Approximate problem}
For $n\in\mathbb{N}$, $\alpha> 0$ and a nonnegative function $f\in L^{1}(\Omega)\setminus\{0\}$, define
\[
f_n(x):=\min\{f(x),n\},
\]
and consider the following approximated problem:
\begin{equation}\label{approxeqn}
M_\alpha\,u = f_n(x)\Big(u^{+}+\frac{1}{n}\Big)^{-\delta(x)} \quad \text{in }\Omega, 
\qquad 
u=0 \quad \text{in }\mathbb{H}^{N}\setminus\Omega.
\end{equation}
The main objective of this section is to show the existence of $u_n$ and to derive suitable a priori estimates.

\begin{Lemma}\label{approx}
Let $\delta:\overline{\Omega}\to(0,\infty)$ be a continuous function. Then, for each $n\in\mathbb{N}$, the problem \eqref{approxeqn} admits a unique positive solution 
\[
u_n\in HW_{0}^{1,p}(\Omega)\cap L^{\infty}(\Omega).
\]
Furthermore, the sequence $\{u_n\}$ is monotone increasing, i.e., $u_{n+1}\geq u_n$ in $\Omega$ for all $n\in\mathbb{N}$. In addition, for every $n\in\mathbb{N}$ and every compact set $\omega\Subset\Omega$, there exists a constant $C(\omega)>0$, independent of $n$, such that
\[
u_n \geq C(\omega) > 0 \quad \text{in }\omega.
\]
\end{Lemma}
\begin{proof}
We denote the space $HW_{0}^{1,p}(\Omega)$ by $V$, and let $V^{*}$ be its dual space.  
Define the operator $S:V\to V^{*}$ by
\begin{align*}
\langle S(v),\varphi\rangle 
&:= \int_{\Omega} |\nabla_H v|^{p-2}\na_H v\na_H\varphi\,dx
+ \alpha \int_{\mathbb{H}^{N}}\int_{\mathbb{H}^{N}} 
J_{p}(v(x)-v(y))\big(\varphi(x)-\varphi(y)\big)\,d\mu\\
&\qquad- \int_{\Omega} f_n(x)\Big(v^{+}+\frac{1}{n}\Big)^{-\delta(x)} \varphi\,dx,
\end{align*}
for all $v,\varphi\in V$.  
Using Lemma~\ref{emb} together with Hölder's inequality, it is straightforward to verify that the mapping $S$ is well defined.

\begin{itemize}
    \item \textbf{Coercivity:}  
    By Lemma~\ref{emb} and Hölder's inequality, we have
    \[
    \langle S(v), v\rangle
    = \|v\|^{p}
    - \int_{\Omega} f_n(x)\Big(v^{+}+\frac{1}{n}\Big)^{-\delta(x)} v\,dx
    \geq \|v\|^{p} - C\|v\|,
    \]
    for some constant $C>0$.  
    Since $1<p<Q$, it follows that $S$ is coercive.

\item \textbf{Demicontinuity:}  
Let $v_k, v \in HW_{0}^{1,p}(\Omega)$ be such that $\|v_k - v\| \to 0$ as $k \to \infty$.  
Then, up to a subsequence (still denoted by $v_k$), we have
\[
v_k \to v \quad \text{a.e. in }\Omega,
\]
and the sequence
\begin{equation}\label{bdd1}
\left\{
\frac{J_p\big(v_k(x)-v_k(y)\big)}{|y^{-1}\circ x|^{\frac{Q+s p}{p'}}}
\right\}_{k\in\mathbb{N}}
\quad \text{is bounded in } L^{p'}(\mathbb{H}^{2N}).
\end{equation}
The pointwise convergence of $v_k$ to $v$ yields
\[
\frac{J_p(v_k(x)-v_k(y))}{|y^{-1}\circ x|^{\frac{Q+s p}{p'}}}
\;\to\;
\frac{J_p(v(x)-v(y))}{|y^{-1}\circ x|^{\frac{Q+s p}{p'}}}
\quad \text{a.e.\ in }\mathbb{H}^{2N}.
\]
Using this fact and the boundedness in \eqref{bdd1}, we conclude (up to a subsequence) that
\[
\frac{J_p(v_k(x)-v_k(y))}{|y^{-1}\circ x|^{\frac{Q+s p}{p'}}}
\rightharpoonup
\frac{J_p(v(x)-v(y))}{|y^{-1}\circ x|^{\frac{Q+s p}{p'}}}
\quad \text{weakly in } L^{p'}(\mathbb{H}^{2N}).
\]
Since
\[
\frac{\varphi(x)-\varphi(y)}{|y^{-1}\circ x|^{\frac{Q+s p}{p}}} \in L^{p}(\mathbb{H}^{2N}),
\]
we deduce that
\begin{equation}\label{demi1}
\lim_{k\to\infty}
\int_{\mathbb{H}^{N}}\int_{\mathbb{H}^{N}}
J_p(v_k(x)-v_k(y))\,(\varphi(x)-\varphi(y))\, d\mu
=
\int_{\mathbb{H}^{N}}\int_{\mathbb{H}^{N}}
J_p(v(x)-v(y))\,(\varphi(x)-\varphi(y))\, d\mu,
\end{equation}
for all $\varphi \in HW_{0}^{1,p}(\Omega)$.

Next, the strong convergence
\[
\nabla_H v_k \to \nabla_H v \quad\text{in } L^{p}(\Omega)
\]
implies the pointwise a.e.\ convergence
\[
v_k(x) \to v(x), \qquad \nabla_H v_k(x)\to\nabla_H v(x),
\]
for a.e.\ $x\in\Omega$. Therefore, it follows that
\[
|\nabla_H v_k(x)|^{p-2}\nabla_H v_k(x)
\;\to\;
|\nabla_H v(x)|^{p-2}\nabla_H v(x)
\quad\text{a.e.\ in }\Omega.
\]
Moreover, the uniform bound
\[
\|\,|\nabla_H v_k|^{p-1}\,\|_{L^{p'}(\Omega)}^{p'}
=
\int_{\Omega} |\nabla_H v_k|^{p}\, dx
\le C,
\]
for some constant $C>0$ independent of $k$, implies (up to a subsequence)
\[
|\nabla_H v_k|^{p-2}\nabla_H v_k
\rightharpoonup
|\nabla_H v|^{p-2}\nabla_H v
\quad\text{weakly in } L^{p'}(\Omega).
\]
Thus, for any $\varphi\in V$,
\begin{equation}\label{eq:limit_identity}
\lim_{k\to\infty}
\int_{\Omega}
|\nabla_H v_k|^{p-2}\nabla_H v_k  \nabla_H\varphi\, dx
=
\int_{\Omega}
|\nabla_H v|^{p-2}\nabla_H v  \nabla_H\varphi\, dx.
\end{equation}

Finally, by the dominated convergence theorem,
\begin{equation}\label{dctap1}
\lim_{k\to\infty}
\int_{\Omega}
f_n(x)\Big(v_k^{+}+\frac{1}{n}\Big)^{-\delta(x)}\varphi\, dx
=
\int_{\Omega}
f_n(x)\Big(v^{+}+\frac{1}{n}\Big)^{-\delta(x)} \varphi\, dx,
\qquad \forall\,\varphi\in HW_{0}^{1,p}(\Omega).
\end{equation}

Combining \eqref{demi1}, \eqref{eq:limit_identity}, and \eqref{dctap1}, we conclude that
\[
\lim_{k\to\infty} \langle S(v_k), \varphi\rangle
=
\langle S(v), \varphi\rangle,
\quad \forall\, \varphi\in HW_{0}^{1,p}(\Omega),
\]
and therefore the operator $S$ is demicontinuous.

\item \textbf{Monotonicity of $S$:}  
Let $u_1, u_2 \in V$. We compute
\begin{align*}
&\langle S(u_1)-S(u_2), u_1-u_2\rangle\\
&= \int_{\Omega} \big(|\nabla_H u_1(x)|^{p-2}\na_H u_1(x) - |\na_H u_2(x)|^{p-2}\na_H u_2(x)\big)\,(\nabla_H u_1(x)-\nabla_H u_2(x))\,dx \\
&\quad + \alpha\int_{\mathbb{H}^{N}}\int_{\mathbb{H}^{N}}
\big(J_p(u_1(x)-u_1(y)) - J_p(u_2(x)-u_2(y))\big)\big((u_1-u_2)(x)-(u_1-u_2)(y)\big)\, d\mu \\
&\qquad - \int_{\Omega} f_n(x)\Big[\Big(u_1^{+}+\frac{1}{n}\Big)^{-\delta(x)}
 - \Big(u_2^{+}+\frac{1}{n}\Big)^{-\delta(x)}\Big](u_1-u_2)\,dx.
\end{align*}
By Lemma~\ref{alg}, the first integral is nonnegative.  
A direct check shows that the second integral is nonpositive.  
Hence $S$ is monotone.

\end{itemize}

By Theorem~\ref{MB}, the operator $S$ is surjective. Therefore, for each $n\in\mathbb{N}$,  
there exists $u_n \in V$ such that
\begin{equation}\label{auxeqn}
\begin{split}
&\int_{\Omega} |\nabla_H u_n(x)|^{p-2}\na_H u_n(x)\nabla_H\varphi\,dx+ \alpha \int_{\mathbb{H}^{N}}\int_{\mathbb{H}^{N}}
J_p\big(u_n(x)-u_n(y)\big)\big(\varphi(x)-\varphi(y)\big)\,d\mu\\
&\quad=
\int_{\Omega} f_n(x) \Big(u_n^{+}+\frac{1}{n}\Big)^{-\delta(x)} \varphi\,dx,
\end{split}
\end{equation}
for all $\varphi \in V$.

\medskip
\noindent \textbf{Nonnegativity.}
Choosing $\varphi = (u_n)_{-} := \min\{u_n, 0\}$ in \eqref{auxeqn}, we obtain
\begin{equation}\label{posi}
\begin{aligned}
&\int_{\Omega} |\nabla_H u_n(x)|^{p-2}\na_H u_n(x)\nabla_H (u_n)_{-}\,dx
+\alpha \int_{\mathbb{H}^{N}}\int_{\mathbb{H}^{N}}
J_p(u_n(x)-u_n(y))\big((u_n)_{-}(x)-(u_n)_{-}(y)\big)\,d\mu \\
&\qquad = \int_{\Omega}
f_n(x)\Big(u_n^{+}+\frac{1}{n}\Big)^{-\delta(x)} (u_n)_{-}\,dx
\;\le 0.
\end{aligned}
\end{equation}

For any $x,y\in\mathbb{H}^{N}$, estimate (3.13) in \cite[page~12]{GU21} gives
\begin{equation}\label{pos}
J_p(u_n(x)-u_n(y))
\big((u_n)_{-}(x)-(u_n)_{-}(y)\big)
\ge 
\big|(u_n)_{-}(x)-(u_n)_{-}(y)\big|^{p}.
\end{equation}
Using \eqref{pos} in \eqref{posi}, and applying Lemma~\ref{alg}, we deduce that
\[
\|(u_n)_{-}\|^{p} = 0,
\]
which implies $(u_n)_{-}=0$. Therefore,
\begin{equation}\label{non-neg}
u_n \ge 0 \quad \text{in }\mathbb{H}^{N},\qquad \text{for every } n\in\mathbb{N}.
\end{equation}

\textbf{Monotonicity and Uniqueness.}  
Fix $n\in\mathbb{N}$ and let $u_n, u_{n+1} \in HW_{0}^{1,p}(\Omega)$ be the corresponding
solutions to \eqref{approxeqn}. By \eqref{non-neg}, both $u_n$ and $u_{n+1}$ are nonnegative in
$\mathbb{H}^N$. For every $\varphi\in HW_{0}^{1,p}(\Omega)$, they satisfy
\begin{equation}\label{auxeqn11}
\begin{split}
&\int_{\Omega} |\nabla_H u_n(x)|^{p-2}\na_H u_n(x)\nabla_H\varphi\,dx
+\alpha\int_{\mathbb{H}^{N}}\int_{\mathbb{H}^{N}}
J_p(u_n(x)-u_n(y))(\varphi(x)-\varphi(y))\,d\mu\\
&=
\int_{\Omega} f_n(x)\Big(u_n+\frac{1}{n}\Big)^{-\delta(x)}\varphi\,dx,
\end{split}
\end{equation}
and
\begin{equation}\label{auxeqn21}
\begin{split}
&\int_{\Omega} |\nabla_H u_{n+1}(x)|^{p-2}\na_H u_{n+1}(x)\nabla_H\varphi\,dx
+\alpha\int_{\mathbb{H}^{N}}\int_{\mathbb{H}^{N}}
J_p(u_{n+1}(x)-u_{n+1}(y))(\varphi(x)-\varphi(y))\,d\mu\\
&=
\int_{\Omega} f_{n+1}(x)\Big(u_{n+1}+\frac{1}{n+1}\Big)^{-\delta(x)}\varphi\,dx.
\end{split}
\end{equation}
Choosing $\varphi = w := (u_n - u_{n+1})^{+} \in HW_{0}^{1,p}(\Omega)$ in both
\eqref{auxeqn11}–\eqref{auxeqn21} and subtracting, we obtain, using $f_n \le f_{n+1}$ a.e.,
\begin{align*}
&\int_{\Omega} f_n(x)\Big(u_n+\tfrac{1}{n}\Big)^{-\delta(x)}w\,dx
-
\int_{\Omega} f_{n+1}(x)\Big(u_{n+1}+\tfrac{1}{n+1}\Big)^{-\delta(x)} w\,dx \\
&\qquad\le 
\int_{\Omega} f_{n+1}(x)\,w\,
\frac{
\Big(u_{n+1}+\tfrac{1}{n+1}\Big)^{\delta(x)}
-
\Big(u_n+\tfrac{1}{n}\Big)^{\delta(x)}
}{
\Big(u_n+\tfrac{1}{n}\Big)^{\delta(x)}
\Big(u_{n+1}+\tfrac{1}{n+1}\Big)^{\delta(x)}
}\,dx
\le 0.
\end{align*}
Thus,
\begin{equation}\label{in-neg}
\begin{split}
&\int_{\Omega} (|\nabla_H u_n(x)|^{p-2}\na_H u_n(x)-|\nabla_H u_{n+1}(x)|^{p-2}\na_H u_{n+1}(x))\nabla_H w\,dx\\
&\quad+
\int_{\mathbb{H}^{N}}\int_{\mathbb{H}^{N}}
\big(J_p(u_n(x)-u_n(y))-J_p(u_{n+1}(x)-u_{n+1}(y))\big)(w(x)-w(y))\,d\mu
\le 0.
\end{split}
\end{equation}
Arguing as in the proof of \cite[Lemma~9]{LL}, one has
\[
\big(J_p(u_n(x)-u_n(y)) - J_p(u_{n+1}(x)-u_{n+1}(y))\big)(w(x)-w(y)) \ge 0,
\quad\text{for a.e. } (x,y)\in\mathbb{H}^{2N}.
\]
Combining this with \eqref{in-neg} yields
\[
\int_{\Omega} (|\nabla_H u_n(x)|^{p-2}\na_H u_n(x)-|\nabla_H u_{n+1}(x)|^{p-2}\na_H u_{n+1}(x))
\nabla_H (u_n-u_{n+1})^{+}\,dx \le 0.
\]
Applying Lemma~\ref{alg}, we conclude that
\[
u_{n+1} \ge u_n \quad\text{in }\Omega.
\]
Uniqueness follows from an analogous argument.

\noindent
\textbf{Boundedness:}
To prove the boundedness of \(u_n\), for any \(k \geq 1\) define
\[
A(k):=\{x\in \Omega : u_n(x)\geq k\}.
\]
Choose \(\varphi_k:=(u_n-k)^+=\max\{u_n-k,0\}\) as a test function in \eqref{approxeqn}. 
Using the estimate (3.9) from \cite[page 11]{GU21}, we obtain
\begin{equation}\label{tstbd1}
\begin{split}
\|\varphi_k\|^{p}\leq C \int_{\Omega} f_n(x)\Big(u_n+\frac{1}{n}\Big)^{-\delta(x)}\,\varphi_k\,dx,
\end{split}
\end{equation}
for some constant $C>0$ independent of $k$.
Therefore, applying Lemma \ref{emb} and using the continuity of the embedding 
\(HW_0^{1,p}(\Omega)\hookrightarrow L^{l}(\Omega)\) for some \(l>p\), we obtain
\begin{multline}\label{tstbd2}
\|\varphi_k\|^{p}
   \leq C\int_{\Omega} n^{\delta(x)+1}\,\varphi_k\,dx
   \leq C\|n^{\delta(x)+1}\|_{L^\infty(\Omega)}
        \int_{A(k)} (u_n - k)\,dx
   \leq C\,|A(k)|^{\frac{l-1}{l}} \|\varphi_k\|,
\end{multline}
for some constant \(C>0\) depending on \(n\). Hence, we obtain
\begin{equation}\label{new}
    \|\varphi_k\|^{p} \leq C\,|A(k)|^{\frac{p(l-1)}{l(p-1)}},
\end{equation}
for some constant \(C>0\) depending on \(n\).  
Now choose \(h>k\) with \(h\geq 1\).  
Since \(u(x)-k \geq h-k\) on \(A(h)\) and \(A(h)\subset A(k)\), using this observation together with \eqref{new}, we deduce
\begin{multline*}
(h-k)^p\,|A(h)|^{\frac{p}{l}} 
\leq \left(\int_{A(h)} (u_n(x)-k)^l\,dx\right)^{\frac{p}{l}}
\leq \left(\int_{A(k)} (u_n(x)-k)^l\,dx\right)^{\frac{p}{l}} 
\leq C\,\|\varphi_k\|^p 
\leq C\,|A(k)|^{\frac{p(l-1)}{l(p-1)}},
\end{multline*}
for some constant \(C>0\) depending on \(n\).  
Consequently, we deduce
\[
|A(h)| \leq \frac{C}{(h-k)^l}\,|A(k)|^{\frac{l-1}{p-1}}.
\]
Notice that \(\frac{l-1}{p-1} > 1\).  
Hence, by applying \cite[Lemma B.1]{Stam}, we conclude
\[
\|u_n\|_{L^\infty(\Omega)} \leq C,
\]
for some constant \(C>0\) depending on \(n\).  
Therefore, we have \(u_n \in L^\infty(\Omega)\).\\
\textbf{Uniform positivity:}  
From \eqref{non-neg}, we have \(u_1 \geq 0\) in \(\mathbb{H}^N\).  
Since \(u_1 = 0\) in \(\mathbb{H}^N \setminus \Omega\) and \(f \not\equiv 0\), it follows that \(u_1 \not\equiv 0\) in \(\Omega\).  
Therefore, by \cite[Theorem 1.4]{Zhang}, for every \(\omega \Subset \Omega\) there exists a constant \(C(\omega) > 0\) such that 
\[
u_1 \geq C(\omega) > 0 \quad \text{in } \omega.
\]  
Using the monotonicity property, we have \(u_n \geq u_1\) in \(\Omega\) for all \(n \in \mathbb{N}\).  
Consequently, for every \(\omega \Subset \Omega\) and every \(n \in \mathbb{N}\),
\[
u_n(x) \geq C(\omega) > 0, \quad x \in \omega,
\]
where \(C(\omega) > 0\) is independent of \(n\).
\end{proof}
\begin{Remark}\label{rmkapprox}
By Lemma \ref{approx}, the monotone sequence \(\{u_n\}\) admits a pointwise limit in \(\Omega\), which we denote by \(u\).  
As a consequence, we have \(u \geq u_n\) in \(\mathbb{H}^N\) for every \(n \in \mathbb{N}\).  
In the following, we will show that \(u\) is indeed a solution to problem \eqref{meqn}.
\end{Remark}

\begin{Remark}\label{altrmk}
We observe that the proof of Lemma \ref{approx} extends naturally to the Euclidean setting.  
Furthermore, Lemma \ref{approx} can also be established via Schauder's fixed point theorem, as illustrated in {\cite[Lemma 3.1]{GU21}}.  
For this approach, the result stated in Lemma {\ref{auxresult}}—which deals with the more general non-singular case—will be useful.  
Its proof follows closely the argument of Lemma \ref{approx}, with the main modification being the consideration of the mapping 
\[
S: V \to V^*, \qquad 
\langle S(v), \varphi \rangle = \int_{\Om}|\na_H v|^{p-2}\na_H\varphi\,dx+\alpha\int_{\mathbb{H}^N} \int_{\mathbb{H}^N} J_p(v(x)-v(y)) (\varphi(x)-\varphi(y)) \, d\mu - \int_{\Omega} g \varphi \, dx,
\]
for all \(v, \varphi \in V\), where \(V = HW_0^{1,p}(\Omega)\) and \(V^*\) denotes its dual.
\end{Remark}

\begin{Lemma}\label{auxresult}
Let $g \in L^{\infty}(\Omega) \setminus \{0\}$ be a nonnegative function in $\Omega$.  
Then, there exists a unique solution $u \in HW_0^{1,p}(\Omega) \cap L^{\infty}(\Omega)$ to the problem
\begin{equation}\label{approxnew}
M_\alpha\,u = g \text{ in } \Omega, \quad u > 0 \text{ in } \Omega, \quad u = 0 \text{ in } \mathbb{H}^N \setminus \Omega.
\end{equation}
Moreover, for every $\omega \Subset \Omega$, there exists a constant $C(\omega) > 0$ such that 
\[
u \geq C(\omega) \quad \text{in } \omega.
\]
\end{Lemma}

\subsection{A-priori estimates of the approximate solutions}
We proceed to derive a sequence of boundedness estimates (Lemmas \ref{apun1}-\ref{apun2}) for the sequence of positive solutions $\{u_n\}_{n \in \mathbb{N}}$ of problem \eqref{approxeqn}, as guaranteed by Lemma \ref{approx}. These estimates are essential for establishing the existence and regularity of solutions.
\begin{Lemma}\label{apun1} (Variable singular exponent)
Let $\delta:\overline{\Omega}\to(0,\infty)$ be a continuous function satisfying the condition $(P_{\epsilon,\delta_*})$ for some $\epsilon>0$ and some $\delta_*>0$. Let $f\in L^m(\Omega)\setminus\{0\}$ be a nonnegative function, where 
\[
m = \Big(\frac{(\delta_*+p-1)p^*}{p \delta_*}\Big)^{'}.
\] 
Assume that $\{u_n\}_{n\in\mathbb{N}}$ is the sequence of solutions to \eqref{approx} provided by Lemma \ref{approx}.
\begin{enumerate}
    \item[$(i)$] If $\delta_*=1$, then the sequence $\{u_n\}_{n\in\mathbb{N}}$ is uniformly bounded in $HW_0^{1,p}(\Omega)$.
    \item[$(ii)$] If $\delta_*>1$, then the sequence $\left\{u_n^{\frac{\delta_*+p-1}{p}}\right\}_{n\in\mathbb{N}}$ is uniformly bounded in $HW_0^{1,p}(\Omega)$.
\end{enumerate}
\end{Lemma}
\begin{proof}
\begin{enumerate}
\item[$(i)$] Taking $\varphi = u_n$ as a test function in the weak formulation of \eqref{approxeqn}, we obtain
\begin{equation}\label{api}
\begin{split}
\|u_n\|^p 
&\leq \int_{\Omega} f_n \Big(u_n + \frac{1}{n}\Big)^{-\delta(x)} u_n \, dx \\
&\leq \int_{\overline{\Omega_\epsilon} \cap \{0 < u_n \leq 1\}} f u_n^{1-\delta(x)} \, dx 
+ \int_{\overline{\Omega_\epsilon} \cap \{u_n > 1\}} f u_n^{1-\delta(x)} \, dx 
+ \int_{\omega_\epsilon} f \, \|C(\omega_\epsilon)^{-\delta(x)}\|_{L^\infty(\Omega)} u_n \, dx \\
&\leq \|f\|_{L^1(\Omega)} + \Big( 1 + \|C(\omega_\epsilon)^{-\delta(x)}\|_{L^\infty(\Omega)} \Big) \int_{\Omega} f u_n \, dx,
\end{split}
\end{equation}
where we have used that $0 < \delta(x) \leq 1$ for all $x \in \Omega_\epsilon$ and the property $u_n \geq C(\omega_\epsilon) > 0$ in $\omega_\epsilon$ from Lemma \ref{approx}. 
Since $f \in L^m(\Omega)$ with $m = (p^*)'$, applying H\"older's inequality and Lemma \ref{emb} to \eqref{api}, we obtain
\begin{equation*}
\begin{split}
\|u_n\|^p &\leq \|f\|_{L^1(\Omega)} + \Big(1 + \|C(\omega_\epsilon)^{-\delta(x)}\|_{L^\infty(\Omega)}\Big) \|f\|_{L^m(\Omega)} \|u_n\|_{L^{p^*}(\Omega)} \\
&\leq \|f\|_{L^1(\Omega)} + C \|f\|_{L^m(\Omega)} \|u_n\|,
\end{split}
\end{equation*}
for some positive constant $C$ independent of $n$. 
Hence, the sequence $\{u_n\}_{n \in \mathbb{N}}$ is uniformly bounded in $HW_0^{1,p}(\Omega)$.
\item[$(ii)$] Choosing $u_n^{\delta_*}$ as a test function in the weak formulation of \eqref{approx}, we get
\begin{equation}\label{apigrt1}
\begin{split}
\Big\|u_n^{\frac{\delta_*+p-1}{p}}\Big\|^p 
&\leq c \int_{\Omega} f_n \Big(u_n + \frac{1}{n}\Big)^{-\delta(x)} u_n^{\delta_*} \, dx \\
&\leq c \int_{\overline{\Omega_\epsilon} \cap \{0 < u_n \leq 1\}} f u_n^{\delta_* - \delta(x)} \, dx 
+ c \int_{\overline{\Omega_\epsilon} \cap \{u_n > 1\}} f u_n^{\delta_* - \delta(x)} \, dx \\
&\quad + c \int_{\omega_\epsilon} f \| C(\omega_\epsilon)^{-\delta(x)} \|_{L^\infty(\Omega)} u_n^{\delta_*} \, dx \\
&\leq c \|f\|_{L^1(\Omega)} + c \Big( 1 + \| C(\omega_\epsilon)^{-\delta(x)} \|_{L^\infty(\Omega)} \Big) \int_{\Omega} f u_n^{\delta_*} \, dx \\
&\leq c \|f\|_{L^1(\Omega)} + c \|f\|_{L^m(\Omega)} \Big\| u_n^{\frac{\delta_*+p-1}{p}} \Big\|^{\frac{p \delta_*}{\delta_*+p-1}},
\end{split}
\end{equation}
where $c>0$ is a constant independent of $n$, for some positive constant $C$ independent of $n$, where we have applied Lemma \ref{emb}, used the assumption $\|\delta\|_{L^\infty(\Omega_\epsilon)} \leq \delta_*$, and employed the fact that $u_n \geq C(\omega) > 0$ from Lemma \ref{approx}. Consequently, the sequence 
$\left\{ u_n^{\frac{\delta_*+p-1}{p}} \right\}_{n \in \mathbb{N}}$ is uniformly bounded in $HW_0^{1,p}(\Omega)$.
\end{enumerate}
\end{proof}

\begin{Lemma}\label{apun2} (Constant singular exponent)
Let $\delta:\overline{\Om}\to(0,\infty)$ be a constant function, and let $\{u_n\}_{n\in\mathbb{N}}$ denote the sequence of solutions to \eqref{approx} provided by Lemma \ref{approx}. Assume that $f\in L^m(\Om)\setminus\{0\}$ is nonnegative for some $m$. Then the following hold:
\begin{enumerate}
    \item[$(i)$] If $0<\delta<1$ and $m=\left(\frac{p^{*}}{1-\delta}\right)'$, the sequence $\{u_n\}_{n\in\mathbb{N}}$ is uniformly bounded in $HW_0^{1,p}(\Om)$.
    \item[$(ii)$] If $\delta=1$ and $m=1$, the sequence $\{u_n\}_{n\in\mathbb{N}}$ is uniformly bounded in $HW_0^{1,p}(\Om)$.
    \item[$(iii)$] If $\delta>1$ and $m=1$, the sequence $\Big\{u_n^{\frac{\delta+p-1}{p}}\Big\}_{n\in\mathbb{N}}$ is uniformly bounded in $HW_0^{1,p}(\Om)$.
\end{enumerate}
\end{Lemma}

\begin{proof}
\begin{enumerate}
\item[$(i)$] Taking $u_n$ as a test function in the weak formulation of \eqref{approxeqn}, we get
\[
\|u_n\|^p \leq \int_{\Om} \frac{f_n(x) u_n}{\left(u_n + \frac{1}{n}\right)^\delta}\,dx.
\]
Since $f \in L^m(\Omega)$ and $(1 - \delta)m' = p^*$, applying Lemma \ref{emb} yields
\begin{align*}
\|u_n\|^p &\leq \int_{\Om} f\, u_n^{1-\delta}\, dx 
\leq \|f\|_{L^m(\Om)} \left( \int_{\Om} u_n^{(1-\delta)m'}\, dx \right)^{\frac{1}{m'}} \\
&= \|f\|_{L^m(\Om)} \left( \int_{\Om} u_n^{p^*}\, dx \right)^{\frac{1-\delta}{p^*}} 
\leq C \|f\|_{L^m(\Om)} \|u_n\|^{1-\delta},
\end{align*}
for some constant $C>0$ independent of $n$. Consequently, we obtain
\[
\|u_n\| \leq C,
\]
for some constant $C>0$ independent of $n$. Therefore, the sequence $\{u_n\}_{n\in\mathbb{N}}$ is uniformly bounded in $HW_0^{1,p}(\Om)$.

\item[$(ii)$] By choosing $\varphi = u_n$ as a test function in the weak formulation of \eqref{approxeqn}, we deduce
\begin{equation}\label{}
\|u_n\|^p \leq \int_{\Om} \frac{f_n u_n}{u_n + \frac{1}{n}} \, dx \leq \|f\|_{L^1(\Om)}.
\end{equation}
Therefore, the sequence $\{u_n\}_{n \in \mathbb{N}}$ is uniformly bounded in $HW_0^{1,p}(\Omega)$.

\item[$(iii)$] Choosing $\varphi= u_n^\delta$ as a test function (which belongs to $HW_0^{1,p}(\Omega)$, since $\Phi(s) = s^\delta$, $s \geq 0$, is Lipschitz on bounded intervals), we obtain
\begin{equation}
\label{bound}
\Big\| u_n^{\frac{\delta + p - 1}{p}} \Big\|^p
\leq C \int_{\Omega} f_n(x) \Big(u_n + \frac{1}{n}\Big)^{-\delta} u_n^\delta \, dx 
\leq C \|f\|_{L^1(\Omega)},
\end{equation}
for some constant $C > 0$ independent of $n$. Consequently, it follows from \eqref{bound} that the sequence $\Big\{ u_n^{\frac{\delta + p - 1}{p}} \Big\}_{n \in \mathbb{N}}$ is uniformly bounded in $HW_0^{1,p}(\Omega)$.
\end{enumerate}
\end{proof}
The following gradient convergence theorem is very crucial for us to pass to the limit. Taking into account Lemma \ref{apun1} and Lemma \ref{apun2}, the proof follows the lines of the proof of \cite[Theorem A.1]{GU21}.
\begin{Theorem}[Gradient convergence theorem in the mixed case]\label{thm:A1}
Let $0<s<1<p<\infty$ and let $\delta:\overline{\Omega}\to(0,\infty)$ be a continuous function satisfying the condition $(P_{\varepsilon,\delta_*})$ for some $\varepsilon>0$ and some $\delta_*>0$. Assume that $\{u_n\}_{n\in\mathbb{N}}$ is the sequence of approximate solutions to problem~\eqref{approx} provided by Lemma~\ref{approx}, and let $u$ denote the pointwise limit of $\{u_n\}$.

If $\delta_*\ge 1$, assume further that $f\in L^{m}(\Omega)\setminus\{0\}$ is nonnegative, where
\[
m=\bigg(\frac{(\delta_*+p-1)p^*}{p\,\delta_*}\bigg)'.
\]
If, in addition, $\delta$ is a constant, then assume $f\in L^{m}(\Omega)\setminus\{0\}$ is nonnegative, where
\[
m=\left(\frac{p^*}{1-\delta}\right)' \quad \text{for } \delta\in(0,1),
\]
and $m=1$ for $\delta\ge 1$.

Then, up to a subsequence,
\[
\nabla_H u_n \longrightarrow \nabla_H u \qquad \text{pointwise almost everywhere in } \Omega.
\]
\end{Theorem}

\subsection{Preliminaries for the uniqueness result}
In this section, we assume that $\delta>0$ is a constant function on $\overline{\Omega}$ unless stated otherwise. We first define the real-valued function $g_k$ by
\begin{equation*}
g_k(s) :=
\begin{cases}
\min\{s^{-\delta},\, k\}, & \text{if } s > 0,\\[2mm]
k, & \text{if } s \le 0.
\end{cases}
\end{equation*}
Next, we introduce the function $\Phi_k$, defined as a primitive of $g_k$ and normalized so that $\Phi_k(1) = 0$.

With the above definitions, we consider the functional
\[
J_k : HW^{1,p}_0(\Omega) \longrightarrow [-\infty, +\infty]
\]
given by
\begin{equation*}
J_k(\varphi) := \frac{1}{p} \int_{\Omega} |\nabla_H \varphi|^p \, dx
+ \frac{\alpha}{p} \int_{\mathbb{H}^{N}} \int_{\mathbb{H}^{N}} |\varphi(x)-\varphi(y)|^p \, d\mu
- \int_{\mathbb{H}^{N}} f(x) \, \Phi_k(\varphi) \, dx, \quad \varphi \in HW^{1,p}_0(\Omega),
\end{equation*}
where $f \in L^1(\Omega)\setminus \{0\}$ is nonnegative.

Recall that for a function $z \in HW^{1,p}_{\rm loc}(\Omega) \cap L^{p-1}(\Omega)$ with $z \ge 0$, we say that $z$ is a weak supersolution (resp. subsolution) to \eqref{meqn} if, for every $\varphi \in C^1_c(\Omega)$ with $\varphi \ge 0$, the following inequality holds:
\begin{equation*}
\int_{\Omega} |\nabla_H z|^{p-2} \nabla_H z  \nabla_H \varphi \, dx
+\alpha \int_{\mathbb{H}^{N}} \int_{\mathbb{H}^{N}} J_p(z(x)-z(y)) (\varphi(x) - \varphi(y)) \, d\mu
\;\underset{(\le)}{\ge}\; \int_{\Omega} f(x) z^{-\delta} \varphi \, dx.
\end{equation*}

Given a fixed supersolution $v$, we define $w$ as the minimizer of $J_k$ over the convex set
\[
\mathcal{K} := \{\varphi \in HW^{1,p}_0(\Omega) \,:\, 0 \le \varphi\le v \ \text{a.e. in } \Omega\}.
\]
A direct computation then yields
\begin{equation}\label{eqminxx}
\begin{split}
&\int_{\Omega} |\nabla_H w|^{p-2} \nabla_H w  \nabla_H (\psi - w) \, dx
+ \alpha\int_{\mathbb{H}^{N}} \int_{\mathbb{H}^{N}} J_p(w(x)-w(y))\big((\psi(x)-w(x)) - (\psi(y)-w(y))\big) \, d\mu \\
&\ge \int_{\Omega} f(x) \Phi_k'(w) (\psi - w) \, dx, \quad 
\text{for } \psi \in w + \big(HW^{1,p}_0(\Omega) \cap L^\infty_c(\Omega)\big) \text{ with } 0 \le \psi \le v,
\end{split}
\end{equation}
where $L^\infty_c(\Omega)$ denotes the space of $L^\infty$ functions with compact support in $\Omega$.  
With this notation, the following results can be established by arguments analogous to those in \cite[Lemma 4.1 and Theorem 4.2]{Caninoetal}.  
For clarity, we present the proof below.

\begin{Lemma}\label{lemmause}
Let $\varphi \in C_c^1(\Omega)$ with $\varphi \geq 0$. Then the function $w$ satisfies
\begin{equation}\label{ksjfskgfdfjigf}
\int_{\Omega} |\nabla_H w|^{p-2} \nabla_H w  \nabla_H \varphi \, dx 
+\alpha \int_{\mathbb{H}^N} \int_{\mathbb{H}^N} J_p(w(x)-w(y)) (\varphi(x) - \varphi(y)) \, d\mu
\geq \int_{\Omega} f(x) \, \Phi_k'(w) \, \varphi \, dx.
\end{equation}
\end{Lemma}

\begin{proof}
Let us consider a real-valued function $g \in C_c^\infty(\mathbb{R})$ such that $0 \leq g(t) \leq 1$, $g(t) = 1$ for $t \in [-1,1]$, and $g(t) = 0$ for $t \in (-\infty,-2] \cup [2,\infty)$. Then, for any nonnegative $\varphi \in C_c^1(\Omega)$, define
\[
\varphi_h := g\Big(\frac{w}{h}\Big)\varphi \quad \text{and} \quad \varphi_{h,t} := \min\{w + t \varphi_h, \, v\},
\]
where $h \geq 1$ and $t > 0$.  
Since $\varphi_{h,t} \in w + \big(HW_0^{1,p}(\Omega) \cap L_c^\infty(\Omega)\big)$ and $0 \leq \varphi_{h,t} \leq v$, it follows from \eqref{eqminxx} that
\[
\begin{split}
&\int_{\Omega} |\nabla_H w|^{p-2} \nabla_H w  \nabla_H (\varphi_{h,t} - w) \, dx \\
&\quad + \int_{\mathbb{H}^N} \int_{\mathbb{H}^N} J_p(w(x)-w(y)) \big( (\varphi_{h,t}(x) - w(x)) - (\varphi_{h,t}(y) - w(y)) \big) \, d\mu \\
&\quad \geq \int_{\Omega} f(x) \, \Phi_k'(w) (\varphi_{h,t} - w) \, dx.
\end{split}
\]
By standard manipulations, and using \eqref{eqminxx} together with Lemma \ref{alg}, we deduce for $c=c(p)>0$ that
\[
\begin{split}
\mathbb{I}_1 := & \, c \int_{\Omega} (|\nabla_H \varphi_{h,t}| + |\nabla_H w|)^{p-2} |\nabla_H (\varphi_{h,t} - w)|^2 \, dx \\
& \quad + c\alpha \int_{\mathbb{H}^N} \int_{\mathbb{H}^N} \big(|\varphi_{h,t}(x) - \varphi_{h,t}(y)| + |w(x) - w(y)|\big)^{p-2} \\
& \qquad \times \big((\varphi_{h,t}(x) - w(x)) - (\varphi_{h,t}(y) - w(y))\big)^2 \, d\mu \\
& \le \int_{\Omega} \big(|\nabla_H \varphi_{h,t}|^{p-2} \nabla_H \varphi_{h,t} - |\nabla_H w|^{p-2} \nabla_H w\big) \nabla_H (\varphi_{h,t} - w) \, dx \\
& \quad + \alpha \int_{\mathbb{H}^N} \int_{\mathbb{H}^N} J_p(\varphi_{h,t}(x)-\varphi_{h,t}(y))\big((\varphi_{h,t}(x) - w(x)) - (\varphi_{h,t}(y) - w(y))\big) \, d\mu \\
& \quad - \alpha \int_{\mathbb{H}^N} \int_{\mathbb{H}^N} J_p(w(x)-w(y)) \big((\varphi_{h,t}(x) - w(x)) - (\varphi_{h,t}(y) - w(y))\big) \, d\mu \\
& \le \int_{\Omega} |\nabla_H\varphi_{h,t}|^{p-2} \nabla_H \varphi_{h,t}  \nabla_H (\varphi_{h,t} - w) \, dx \\
& \quad + \alpha \int_{\mathbb{H}^N} \int_{\mathbb{H}^N}J_p(\varphi_{h,t}(x)-\varphi_{h,t}(y)) \big((\varphi_{h,t}(x) - w(x)) - (\varphi_{h,t}(y) - w(y))\big) \, d\mu \\
& \quad - \int_{\Omega} f(x) \, \Phi_k'(w) (\varphi_{h,t} - w) \, dx.
\end{split}
\]
We can rewrite this as
\begin{equation}\label{cicciofriccio}
\begin{split}
\mathbb{I}_1 &- \int_\Omega f(x)\big(\Phi_k'(\varphi_{h,t}) - \Phi_k'(w)\big) (\varphi_{h,t} - w) \, dx \\
&\leq \int_{\Omega} |\nabla_H \varphi_{h,t}|^{p-2} \nabla_H \varphi_{h,t}  \nabla_H (\varphi_{h,t}-w) \, dx \\
&\quad + \alpha \int_{\mathbb{H}^N} \int_{\mathbb{H}^N} J_p(\varphi_{h,t}(x)-\varphi_{h,t}(y))\big((\varphi_{h,t}(x) - w(x)) - (\varphi_{h,t}(y) - w(y))\big) \, d\mu \\
&\quad - \int_\Omega f(x) \, \Phi_k'(\varphi_{h,t}) (\varphi_{h,t} - w) \, dx \\
&= \int_{\Omega} g(x) \, dx + \alpha \int_{\mathbb{H}^N} \int_{\mathbb{H}^N} \mathcal{G}(x,y) \, dx \, dy - \int_\Omega f(x) \, \Phi_k'(\varphi_{h,t}) (\varphi_{h,t} - w - t \varphi_h) \, dx \\
&\quad + t \int_{\Omega} |\nabla_H \varphi_{h,t}|^{p-2} \nabla_H \varphi_{h,t} \nabla_H \varphi_h \, dx+t\alpha \int_{\mathbb{H}^N} \int_{\mathbb{H}^N} J_p(\varphi_{h,t}(x) - \varphi_{h,t}(y))(\varphi_h(x) - \varphi_h(y)) \, d\mu \\
&\quad - t \int_\Omega f(x) \, \Phi_k'(\varphi_{h,t}) \varphi_h \, dx,
\end{split}
\end{equation}
where, we set
\[
G(x) := |\na_H\varphi_{h,t}|^{p-2}\na_H\varphi_{h,t}\nabla_H (\varphi_{h,t} - w - t \varphi_h),
\]
and
\[
\mathcal{G}(x,y) := \frac{J_p(\varphi_{h,t}(x) - \varphi_{h,t}(y)) \left[ (\varphi_{h,t}(x) - w(x) - t \varphi_h(x)) - (\varphi_{h,t}(y) - w(y) - t \varphi_h(y)) \right]}{|y^{-1} \circ x|^{Q + sp}}.
\]

For later use, we also define
\[
G_v(x) := |\na_H v|^{p-2}\na_H v \, \nabla_H (\varphi_{h,t} - w - t \varphi_h),
\]
and
\[
\mathcal{G}_v(x,y) := \frac{J_p(v(x) - v(y)) \left[ (\varphi_{h,t}(x) - w(x) - t \varphi_h(x)) - (\varphi_{h,t}(y) - w(y) - t \varphi_h(y)) \right]}{|y^{-1} \circ x|^{Q + sp}}.
\]
Then, following exactly the same arguments as in the proofs of \cite[Lemma 4.6]{Garainmn} and \cite[Lemma 4.1]{Caninoetal}, we obtain
\[
\int_{\Omega} G(x) \, dx = \int_{\Omega} G_v(x) \, dx
\quad \text{and} \quad
\int_{\mathbb{R}^N} \int_{\mathbb{R}^N} \mathcal{G}(x,y) \, dx \, dy
=
\int_{\mathbb{R}^N} \int_{\mathbb{R}^N} \mathcal{G}_v(x,y) \, dx \, dy.
\]
Returning to \eqref{cicciofriccio}, we obtain
\[
\begin{split}
\mathbb{I}_1 &- \int_\Omega f(x)\, (\Phi_k'(\varphi_{h,t}) - \Phi_k'(w)) (\varphi_{h,t} - w) \, dx \\
&\leq \int_{\mathbb{H}^{N}} \int_{\mathbb{H}^{N}} \mathcal{G}_v(x,y) \, dx \, dy 
- \int_\Omega f(x) \, \Phi_k'(\varphi_{h,t}) (\varphi_{h,t} - w - t \varphi_h) \, dx \\
&\quad + t\alpha \int_{\mathbb{H}^{N}} \int_{\mathbb{H}^{N}} J_p(\varphi_{h,t}(x) - \varphi_{h,t}(y)) (\varphi_h(x) - \varphi_h(y)) \, d\mu \\
&\quad - t \int_\Omega f(x) \, \Phi_k'(\varphi_{h,t}) \, \varphi_h \, dx.
\end{split}
\]
Using that $\varphi_{h,t} - w - t\varphi_h \leq 0$ and that $v$ is a weak supersolution to $Mz = \Phi_k'(z)$, we obtain
\[
\begin{split}
\mathbb{I}_1 &- \int_\Omega f(x)\, (\Phi_k'(\varphi_{h,t}) - \Phi_k'(w)) (\varphi_{h,t} - w) \, dx \\
&\leq t \alpha \int_{\mathbb{H}^N} \int_{\mathbb{H}^N} J_p(\varphi_{h,t}(x) - \varphi_{h,t}(y)) (\varphi_h(x) - \varphi_h(y)) \, d\mu - t \int_\Omega f(x) \, \Phi_k'(\varphi_{h,t}) \, \varphi_h \, dx.
\end{split}
\]
Since $\varphi_{h,t} - w \leq t \, \varphi_h$, we deduce
\[
\begin{split}
&\int_{\Om}|\na_H\varphi_{h,t}|^{p-2}\na_H\varphi_{h,t}\na_H\varphi_h\,dx+\alpha\int_{\mathbb{H}^N} \int_{\mathbb{H}^N} J_p(\varphi_{h,t}(x) - \varphi_{h,t}(y)) (\varphi_h(x) - \varphi_h(y)) \, d\mu \\
&\quad - \int_\Omega f(x) \, \Phi_k'(\varphi_{h,t}) \, \varphi_h \, dx
\geq - \int_\Omega f(x) \, |\Phi_k'(\varphi_{h,t}) - \Phi_k'(w)| \, |\varphi_h| \, dx.
\end{split}
\]
Passing to the limit as $t \to 0$ and applying the Lebesgue dominated convergence theorem, we obtain
\[
\int_{\Om}|\na_H w|^{p-2}\na_H w\na_H\varphi_h\,dx+\alpha\int_{\mathbb{H}^N} \int_{\mathbb{H}^N} J_p(w(x) - w(y)) (\varphi_h(x) - \varphi_h(y)) \, d\mu 
- \int_\Omega f(x) \, \Phi_k'(w) \, \varphi_h \, dx \geq 0.
\]
The desired result, i.e., \eqref{ksjfskgfdfjigf}, then follows by letting $h \to \infty$.
\end{proof}
We are now in a position to establish the \emph{weak comparison principle} as follows:
\begin{Theorem}\label{comparison}
Let $f \in L^1(\Omega) \setminus \{0\}$ be nonnegative. Suppose $u$ is a weak subsolution of \eqref{meqn} satisfying $u \leq 0$ on $\partial \Omega$, and let $v$ be a weak supersolution of \eqref{meqn}. Then,
\[
u \leq v \quad \text{a.e. in } \Omega.
\]
\end{Theorem}

\begin{proof}
For $\varepsilon > 0$ and $w$ as in Lemma \ref{lemmause}, we have
\[
(u - w - \varepsilon)^+ \in HW^{1,p}_0(\Omega).
\]
This follows directly from the fact that $w \in HW^{1,p}_0(\Omega)$ and $w \geq 0$ a.e. in $\Omega$, so that the support of $(u - w - \varepsilon)^+$ is contained in the support of $(u - \varepsilon)^+$.  
Hence, by \eqref{ksjfskgfdfjigf} and standard density arguments, we deduce:
\begin{equation}\label{eq111}
\begin{split}
&\int_{\Omega} |\nabla_H w|^{p-2} \nabla_H w \nabla_H T_\tau\big((u - w - \varepsilon)^+\big) \, dx\\
&\qquad+ \alpha \int_{\mathbb{H}^N} \int_{\mathbb{H}^N} J_p(w(x) - w(y)) \Big( T_\tau\big((u - w - \varepsilon)^+\big)(x) - T_\tau\big((u - w - \varepsilon)^+\big)(y) \Big) \, d\mu \\
&\geq \int_{\Omega} f(x) \, \Phi_k'(w) \, T_\tau\big((u - w - \varepsilon)^+\big) \, dx,
\end{split}
\end{equation}
where $T_\tau(s) := \min\{s, \tau\}$ for $s \geq 0$, and $T_\tau(-s) := -T_\tau(s)$ for $s < 0$.

Let $\varphi_n \in C_c^1(\Omega)$ be such that $\varphi_n \to (u - w - \varepsilon)^+$ in $HW^{1,p}_0(\Omega)$, and define
\[
\tilde{\varphi}_{\tau,n} := T_\tau\Big( \min\{ (u - w - \varepsilon)^+, \, \varphi_n^+ \} \Big).
\]
Then $\tilde{\varphi}_{\tau,n} \in HW^{1,p}_0(\Omega) \cap L_c^\infty(\Omega)$, and by a standard density argument, we have
\begin{equation}\nonumber
\int_{\Omega} |\nabla_H u|^{p-2} \nabla_H u \nabla_H \tilde{\varphi}_{\tau,n} \, dx
+ \int_{\mathbb{H}^N} \int_{\mathbb{H}^N} J_p(u(x) - u(y))\, (\tilde{\varphi}_{\tau,n}(x) - \tilde{\varphi}_{\tau,n}(y)) \, d\mu
\leq \int_\Omega f(x) u^{-\delta} \, \tilde{\varphi}_{\tau,n} \, dx.
\end{equation}
Passing to the limit as $n \to \infty$, we obtain
\begin{equation}\label{eq222}
\begin{split}
&\int_{\Omega} |\nabla_H u|^{p-2} \nabla_H u \nabla_H T_\tau\big((u - w - \varepsilon)^+\big) \, dx\\
&\quad+ \alpha \int_{\mathbb{H}^N} \int_{\mathbb{H}^N} J_p(u(x) - u(y)) \Big( T_\tau\big((u - w - \varepsilon)^+(x)\big) - T_\tau\big((u - w - \varepsilon)^+(y)\big) \Big) \, d\mu \\
&\leq \int_\Omega \frac{f(x)}{u^\delta} \, T_\tau\big((u - w - \varepsilon)^+\big) \, dx.
\end{split}
\end{equation}
Now, define
\[
g(t) := T_\tau\big((t - \varepsilon)^+\big) = \min\{\tau, \max\{t - \varepsilon, 0\}\}.
\]

With this notation, we can write
\begin{equation}\nonumber
\begin{split}
&|u(x) - u(y)|^{p-2} (u(x) - u(y)) \Big( T_\tau\big((u - w - \varepsilon)^+(x)\big) - T_\tau\big((u - w - \varepsilon)^+(y)\big) \Big) \\
&\quad = |u(x) - u(y)|^{p-2} (u(x) - u(y)) \big( (u(x) - w(x)) - (u(y) - w(y)) \big) \, H(x,y),
\end{split}
\end{equation}
where $H(x,y)$ is a suitable function accounting for the truncation with
\[
H(x,y) := \frac{g(u(x) - w(x)) - g(u(y) - w(y))}{(u(x) - w(x)) - (u(y) - w(y))}\,.
\]
Similarly, we have
\begin{equation}\nonumber
\begin{split}
&|w(x) - w(y)|^{p-2} (w(x) - w(y)) \Big( T_\tau\big((u - w - \varepsilon)^+(x)\big) - T_\tau\big((u - w - \varepsilon)^+(y)\big) \Big) \\
&\quad = |w(x) - w(y)|^{p-2} (w(x) - w(y)) \big( (u(x) - w(x)) - (u(y) - w(y)) \big) \, H(x,y).
\end{split}
\end{equation}
Now, subtracting \eqref{eq111} from \eqref{eq222}, choosing $\varepsilon > 0$ such that $\varepsilon^{-\beta} < k$, and applying Lemma \ref{alg}, we obtain
\begin{equation}\nonumber
\begin{split}
&c\int_{\Omega} (|\nabla_H u| + |\nabla_H w|)^{p-2} \big|\nabla_H T_\tau((u - w - \varepsilon)^+)\big|^2 \, dx \\
&\quad + c \alpha \int_{\mathbb{H}^N} \int_{\mathbb{H}^N} (|u(x) - u(y)| + |w(x) - w(y)|)^{p-2} \big( (u(x) - w(x)) - (u(y) - w(y)) \big)^2 H(x,y) \, d\mu \\
&\leq \int_{\Omega} f(x) \Big( u^{-\delta} - \Phi_k'(w) \Big) T_\tau((u - w - \varepsilon)^+) \, dx \\
&\leq \int_{\Omega} f(x) \Big( \Phi_k'(u) - \Phi_k'(w) \Big) T_\tau((u - w - \varepsilon)^+) \, dx \leq 0.
\end{split}
\end{equation}
Hence, we deduce
\[
u \leq w + \varepsilon \leq v + \varepsilon \quad \text{a.e. in } \Omega,
\]
and the desired result follows by letting $\varepsilon \to 0$.
\end{proof}

\section{Proof of the main results}
\subsection{Proof of the existence results}
\textbf{Proof of Theorem \ref{varthm1}:} We prove the theorem separately for the cases $\delta_* = 1$ and $\delta_* > 1$.
\begin{itemize}
    \item For $\delta_* = 1$, by Lemma \ref{apun1}-$(i)$, the sequence of solutions $\{u_n\}_{n \in \mathbb{N}} \subset HW^{1,p}_0(\Omega) \cap L^\infty(\Omega)$ of the problem~\eqref{approxeqn}, provided by Lemma \ref{approx}, is uniformly bounded in $HW^{1,p}_0(\Omega)$. 
    Consequently, by Lemma \ref{emb}, up to a subsequence, we have
    \[
    u_n \rightharpoonup u \quad \text{weakly in } HW^{1,p}_0(\Omega), \quad
    u_n \to u \quad \text{strongly in } L^r(\Omega) \text{ for } 1 \le r < p^*,
    \]
    and $u_n \to u \quad \text{pointwise a.e. in } \Omega$. Moreover, by Lemma~\ref{approx}, we obtain for every $K \Subset \Omega$, there exists a constant $C = C(K) > 0$ \text{ such that } $u(x) \ge C > 0$ for a.e. $x \in K$.
We have
\begin{equation}\label{eq:equazPn}
\begin{split}
&\int_{\Om}|\na_H u_n|^{p-2}\na_H u_n\na_H\varphi\,dx+\alpha \int_{\mathbb{H}^{N}}\int_{\mathbb{H}^{N}} 
\frac{J_p(u_n(x) - u_n(y)) \, (\varphi(x) - \varphi(y))}{|y^{-1} \circ x|^{Q+sp}} \, dx \, dy\\
&=\int_{\Omega} f_n(x) \left(u_n + \frac{1}{n}\right)^{-\delta(x)} \varphi \, dx,
\end{split}
\end{equation}
for all $\varphi \in C^1_c(\Omega)$.

First we observe that the pointwise limit $u \in HW^{1,p}_0(\Omega)$, and therefore $u \in L^{p-1}(\Omega)$. 
Moreover, by {Theorem \ref{thm:A1}}, up to a subsequence, we have
\[
\nabla_H u_n \to \nabla_H u \quad \text{pointwise a.e. in } \Omega.
\]
As a consequence, for every $\varphi \in C^1_c(\Omega)$, it follows that
\begin{equation}\label{loclim}
\lim_{n \to \infty} 
\int_{\Omega} |\nabla_H u_n|^{p-2} \nabla_H u_n \nabla_H \varphi \, dx
=
\int_{\Omega} |\nabla_H u|^{p-2} \nabla_H u \nabla_H \varphi \, dx.
\end{equation}

Moreover, the sequence
\[
\left\{ \frac{J_p(u_n(x) - u_n(y))}{|y^{-1} \circ x|^{\frac{Q+sp}{p'}}} \right\}_{n \in \mathbb{N}} 
\quad \text{is bounded in } L^{p'}(\mathbb{H}^{2N}),
\]
By the pointwise convergence of $u_n$ to $u$, we have
\[
\frac{J_p(u_n(x)-u_n(y))}{|y^{-1} \circ x|^{\frac{Q+sp}{p'}}} 
\to 
\frac{J_p(u(x)-u(y))}{|y^{-1} \circ x|^{\frac{Q+sp}{p'}}} 
\quad \text{pointwise a.e. in } \mathbb{H}^{2N}.
\]

Hence, by standard results, it follows that
\[
\frac{J_p(u_n(x)-u_n(y))}{|y^{-1} \circ x|^{\frac{Q+sp}{p'}}} 
\to 
\frac{J_p(u(x)-u(y))}{|y^{-1} \circ x|^{\frac{Q+sp}{p'}}} 
\quad \text{weakly in } L^{p'}(\mathbb{H}^{2N}).
\]
Next, since for $\varphi \in C^1_c(\Omega)$ we have
\[
\frac{\varphi(x) - \varphi(y)}{|y^{-1}\circ x|^{\frac{Q+sp}{p}}} \in L^p(\mathbb{H}^{2N}),
\]
it follows that
\begin{equation}\label{lhs}
\begin{split}
&\lim_{n \to +\infty} \int_{\mathbb{H}^{N}}\int_{\mathbb{H}^{N}} 
\frac{J_p(u_n(x)-u_n(y)) (\varphi(x)-\varphi(y))}{|y^{-1}\circ x|^{Q+sp}} \, dx\, dy \\
&\quad = \int_{\mathbb{H}^{N}}\int_{\mathbb{H}^{N}} 
\frac{J_p(u(x)-u(y)) (\varphi(x)-\varphi(y))}{|y^{-1}\circ x|^{Q+sp}} \, dx\, dy,
\end{split}
\end{equation}
for all $\varphi\in C^1_c(\Omega)$.

Concerning the right-hand side of \eqref{eq:equazPn}, by Lemma \ref{approx}, for any $\varphi \in C^1_c(\Omega)$ with ${\rm supp}(\varphi) = K$, there exists a constant $C = C(K) > 0$ independent of $n$ such that
\[
\left| f_n(x)\, \varphi \, \Big(u_n + \frac{1}{n}\Big)^{-\delta(x)} \right| \leq \big\| C^{-\delta(x)} \big\|_{L^\infty(\Omega)} \, |f(x) \varphi(x)| \in L^1(\Omega).
\]

By the Lebesgue dominated convergence theorem, we have
\begin{equation}\label{dct}
\lim_{n\to\infty} \int_{\Omega} f_n(x) \Big(u_n + \frac{1}{n}\Big)^{-\delta(x)} \, \varphi \, dx
= \int_{\Omega} f(x)\, u^{-\delta(x)} \, \varphi\, dx.
\end{equation}

Finally, using \eqref{loclim}, \eqref{lhs}, and \eqref{dct} in \eqref{eq:equazPn}, we obtain
\begin{equation*}
\int_{\Omega} |\nabla_H u|^{p-2} \nabla_H u \nabla_H \varphi \, dx
+ \alpha \int_{\mathbb{H}^{N}} \int_{\mathbb{H}^{N}} \frac{J_p(u(x) - u(y)) (\varphi(x) - \varphi(y))}{|y^{-1}\circ x|^{Q + sp}} \, dx \, dy
= \int_{\Omega} f(x)\, u^{-\delta(x)} \, \varphi\, dx,
\end{equation*}
for all $\varphi \in C^1_c(\Omega)$, which shows that $u \in HW_0^{1,p}(\Omega)$ is a weak solution to \eqref{meqn}.

\item Let $\delta_* > 1$. Then, by Lemma \ref{apun1}-$(ii)$, the sequence 
\[
\left\{ u_{n}^{\frac{\delta + p - 1}{p}} \right\}_{n \in \mathbb{N}}
\]
is uniformly bounded in $W^{1,p}_{0}(\Omega)$. Consequently, 
\[
u^{\frac{\delta + p - 1}{p}} \in W^{1,p}_{0}(\Omega).
\] 
Combining this with Lemma \ref{approx}, we also obtain 
\[
u \in L^{p-1}(\Omega).
\]
Now, following the previous step, for every 
$\varphi \in C^{1}_{c}(\Omega)$, we have
\[
\lim_{n \to \infty} 
\int_{\Omega} |\nabla_H u_n|^{p-2} \nabla_H u_n \nabla_H \varphi \, dx
=
\int_{\Omega} |\nabla_H u|^{p-2} \nabla_H u \nabla_H \varphi \, dx,
\]
and
\[
\lim_{n\to\infty}
\int_{\Omega} f_n \left( u_n + \frac{1}{n} \right)^{-\delta} \varphi \, dx
=
\int_{\Omega} f\, u^{-\delta} \varphi \, dx.
\]

For the nonlocal term, proceeding as in the proof of \cite[Theorem 3.6]{Caninoetal}, 
for every $\varphi \in C^{1}_{c}(\Omega)$, we obtain
\[
\lim_{n\to\infty}
\int_{\mathbb{H}^{N}} \int_{\mathbb{H}^{N}}
J_p(u_n(x)-u_n(y))\, (\varphi(x) - \varphi(y))\, d\mu
=
\int_{\mathbb{H}^{N}} \int_{\mathbb{H}^{N}}
J_p(u(x)-u(y))\, (\varphi(x) - \varphi(y))\, d\mu.
\]
Combining the three estimates above, the conclusion immediately follows.
\end{itemize}

\textbf{Proof of Theorem \ref{thm1}:}
\begin{enumerate}
\item[$(i)$] Let $0<\delta<1$. By Lemma \ref{apun2}-$(i)$, the sequence of solutions $\{u_n\}_{n\in \mathbb N} \subset HW^{1,p}_0(\Omega)$ to the problem~\eqref{approxeqn}, provided by Lemma \ref{approx}, is uniformly bounded in $HW^{1,p}_0(\Omega)$. The result then follows by proceeding along the lines of the proof of Theorem \ref{varthm1} for the case $\delta_*=1$.

\item[$(ii)$] Let $\delta=1$. By Lemma \ref{apun2}-$(ii)$, the sequence of solutions $\{u_n\}_{n\in \mathbb N} \subset HW^{1,p}_0(\Omega)$ to the problem~\eqref{approxeqn}, provided by Lemma \ref{approx}, is uniformly bounded in $HW^{1,p}_0(\Omega)$. The result then follows by proceeding along the lines of the proof of Theorem \ref{varthm1} for the case $\delta_*=1$.

\item[$(iii)$] Let $\delta>1$. By Lemma \ref{apun2}-$(iii)$, the sequence 
$\Big\{ u_n^{\frac{\delta+p-1}{p}} \Big\}_{n\in \mathbb N} \subset HW^{1,p}_0(\Omega)$ 
of the problem~\eqref{approxeqn}, provided by Lemma \ref{approx}, is uniformly bounded in $HW^{1,p}_0(\Omega)$. 
The result then follows by proceeding along the lines of the proof of Theorem \ref{varthm1} for the case $\delta_*>1$.
\end{enumerate}
\subsection{Proof of the regularity results}
\textbf{Proof of Theorem \ref{regthm}:} 
\begin{enumerate}
\item[$(i)$] Choosing the test function $\varphi = u_n^\gamma$ (with $\gamma \geq \delta_*$ to be determined later) in the weak formulation of \eqref{approxeqn}, we obtain
\begin{align}\label{M7}
&\int_{\Omega} |\nabla_H u_n|^{p-2} \nabla_H u_n \nabla_H u_n^\gamma \, dx + \int_{\mathbb{H}^N} \int_{\mathbb{H}^N} J_p(u_n(x) - u_n(y)) \left( u_n^\gamma(x) - u_n^\gamma(y) \right)\, d\mu\\
&\quad\leq \int_\Omega v_n^\gamma f_n \left(v_n + \frac{1}{n} \right)^{-\delta(x)} \, dx.
\end{align}
By the condition $(P_{\epsilon,\delta_*})$, we have
\begin{align}\label{M8}
\int_\Omega u_n^\gamma f_n \left(u_n + \frac{1}{n}\right)^{-\delta(x)} \, dx 
&\leq \int_{\Omega \cap \Omega_\epsilon^c} u_n^\gamma f \left(u_n + \frac{1}{n}\right)^{-\delta(x)} \, dx
+ \int_{\{x \in \Omega_\epsilon : u_n \leq 1\}} u_n^\gamma f \left(u_n + \frac{1}{n}\right)^{-\delta(x)} \, dx \nonumber \\
&\quad + \int_{\{x \in \Omega_\epsilon : u_n > 1\}} u_n^\gamma f \left(u_n + \frac{1}{n}\right)^{-\delta(x)} \, dx \nonumber \\
&\leq \Big\| C(\epsilon)^{-\delta(x)} \Big\|_{L^\infty(\Omega)} \int_\Omega u_n^\gamma f \, dx
+ \int_{\{x \in \Omega_\epsilon : u_n \leq 1\}} u_n^{\gamma - \delta(x)} f \, dx \nonumber \\
&\quad + \int_{\{x \in \Omega_\epsilon : u_n > 1\}} u_n^\gamma f \, dx \nonumber \\
&\leq C \left( \int_\Omega u_n^\gamma f \, dx + \int_{\{x \in \Omega_\epsilon : u_n \leq 1\}} f \, dx + \int_{\{x \in \Omega_\epsilon : u_n > 1\}} u_n^\gamma f \, dx \right) \nonumber \\
&\leq C \left( \int_\Omega u_n^\gamma f \, dx + \|f\|_{L^r(\Omega)} \right) \nonumber \\
&\leq C \|f\|_{L^m(\Omega)} \left[ 1 + \left( \int_\Omega |u_n|^{m' \gamma} \, dx \right)^{\frac{1}{r'}} \right],
\end{align}
for some constant $C>0$ independent of $n$. Using Lemma \ref{BPalg} in \eqref{M7} and applying the estimate \eqref{M8}, we deduce
\begin{align*}
    \Big\| u_n^{\frac{\gamma+p-1}{p}} \Big\|^p 
    \leq C \|f\|_{L^r(\Omega)} \left[ 1 + \left( \int_\Omega |u_n|^{m' \gamma} \, dx \right)^{\frac{1}{m'}} \right],
\end{align*}
for some constant $C>0$ independent of $n$. Then, by Lemma \ref{emb}, we obtain
\begin{align}\label{M9}
    \left( \int_\Omega u_n^{\frac{p^*(\gamma+p-1)}{p}} \, dx \right)^{\frac{p}{p^*}} 
    \leq C \|f\|_{L^m(\Omega)} \left[ 1 + \left( \int_\Omega |u_n|^{m' \gamma} \, dx \right)^{\frac{1}{m'}} \right],
\end{align}
for some constant $C>0$ independent of $n$. We choose $\gamma$ such that 
\[
\frac{p^*(\gamma+p-1)}{p} = m' \gamma, \quad \text{i.e.,} \quad \gamma = \frac{N(p-1)(m-1)}{Q-mp}.
\]
Since 
\[
\frac{Q(\delta_*+p-1)}{Q(p-1)+\delta_* p} \leq m < \frac{Q}{p},
\] 
it follows that $\gamma \ge \delta_*$ and $\frac{1}{m'} < \frac{p}{p^*}$. Therefore, inequality \eqref{M9} implies that the sequence $\{u_n\}_{n\in\mathbb{N}}$ is uniformly bounded in $L^{r'\gamma}(\Omega)$.

\item[$(ii)$] Let $k \geq 1$ and define 
\[
A(k) = \{ x \in \Omega : u_n(x) \geq k \}.
\] 
Choosing 
\[
\varphi_k(x) = (u_n - k)^+ \in HW_0^{1,p}(\Omega)
\] 
as a test function in \eqref{approxeqn}, and applying H\"older's inequality with exponents $p^{*'}$ and $p^*$, followed by Young's inequality with exponents $p$ and $p'$, together with Lemma \ref{emb}, we obtain
\begin{multline}
\label{unibdd}
\|\varphi_k\|^p 
\leq \int_{\Omega} \frac{f_n}{\big(u_n + \frac{1}{n}\big)^{\delta(x)}} \varphi_k \, dx
\leq \int_{A(k)} f(x) \varphi_k \, dx \\
\leq \left( \int_{A(k)} f^{(p^*)'} \, dx \right)^{\frac{1}{{(p^*)'}}} 
\left( \int_{\Omega} \varphi_k^{p^*} \, dx \right)^{\frac{1}{p^*}}
\leq C \left( \int_{A(k)} f^{{(p^*)'}} \, dx \right)^{\frac{1}{{(p^*)'}}} \|\varphi_k\| \\
\leq \epsilon \|\varphi_k\|^p + C(\epsilon) \left( \int_{A(k)} f^{{(p^*)'}} \, dx \right)^{\frac{p'}{{(p^*)'}}}.
\end{multline}
To obtain the second inequality above, we also used that, since $u_n \geq k \geq 1$ on $\mathrm{supp}(\varphi)$, we have
\[
\left| \frac{\varphi}{\big(u_n + \frac{1}{n}\big)^{\delta(x)}} \right| \leq \varphi \quad \text{on } \mathrm{supp}(\varphi).
\]
Here, $C$ denotes the Sobolev constant and $C(\epsilon)>0$ is a constant depending on $\epsilon \in (0,1)$. 
Note that $m > \frac{Q}{p}$ implies $m > (p^*)'$. 
Therefore, fixing $\epsilon \in (0,1)$ and using H\"older's inequality with exponents $\frac{m}{{(p^*)'}}$ and $\left(\frac{m}{p^{*'}}\right)'$, we obtain
\begin{align*}
\|\varphi_k\|^p 
&\leq C \left( \int_{A(k)} f^{{(p^*)'}} \, dx \right)^{\frac{p'}{{(p^*)'}}} 
\leq C \left( \int_{A(k)} f^m \, dx \right)^{\frac{p'}{m}} 
|A(k)|^{\frac{p'}{{(p^*)'}} \frac{1}{\left( \frac{m}{{(p^*)'}} \right)'}}.
\end{align*}
Let $h>0$ be such that $1 \leq k < h$. Then $A(h) \subset A(k)$ and for any $x \in A(h)$, we have $u_n(x) \geq h$, so that $u_n(x) - k \geq h - k$ in $A(h)$.  
Using these observations, there exists a constant $C>0$ independent of $n$ such that
\begin{multline*}
(h-k)^p |A(h)|^{\frac{p}{p^*}} 
\leq \left( \int_{A(h)} (u_n - k)^{p^*} \, dx \right)^{\frac{p}{p^*}} 
\leq \left( \int_{A(k)} (u_n - k)^{p^*} \, dx \right)^{\frac{p}{p^*}} \\
\leq C \|\varphi_k\|^p 
\leq C \|f\|_{L^m(\Omega)}^{p'} |A(k)|^{\frac{p'}{{(p^*)'}}\frac{1}{\left(\frac{m}{{(p^*)'}}\right)'}}.
\end{multline*}
Thus, for some constant $C>0$ independent of $n$, we have
\[
|A(h)| \leq C \frac{\|f\|_{L^m(\Omega)}^{\frac{p^*}{p-1}}}{(h-k)^{p^*}} |A(k)|^{\alpha}, \quad \text{where} \quad \alpha = \frac{p^* p'}{p {(p^*)'}}  \frac{1}{\left(\frac{m}{{(p^*)'}}\right)'}.
\]
Since $m > \frac{Q}{p}$, it follows that $\alpha > 1$. Hence, by \cite[Lemma B.1]{Stam}, we obtain
\[
\|u_n\|_{L^\infty(\Omega)} \leq C,
\]
for some positive constant $C$ independent of $n$. Therefore, $u \in L^\infty(\Omega)$.
    \end{enumerate}

\textbf{Proof of Theorem \ref{regthm1}:}
\begin{enumerate}
\item[$(i)$] We observe that 
\begin{itemize}
    \item for $m = \left(\frac{p^{*}}{1-\delta}\right)'$, i.e., $(1-\delta)m' = p^{*}$, we have $\gamma = \frac{(\delta+p-1)m'}{pm' - p^*} = 1$, and
    \item if $m \in \left(\left(\frac{p^{*}}{1-\delta}\right)', \frac{p^{*}}{p^{*}-p}\right)$, then $\gamma = \frac{(\delta+p-1)m'}{pm' - p^*} > 1$.
\end{itemize}

Note that $(p\gamma - p + 1 - \delta)m' = p^* \gamma$. Choosing $\varphi = u_n^{p\gamma - p + 1} \in HW_0^{1,p}(\Omega)$ as a test function in \eqref{approxeqn}, and applying Lemma \ref{BPalg}, we obtain
\begin{align*}
\|u_n^{\gamma}\|^p 
&\leq \|f\|_{L^m(\Omega)} \left( \int_{\Omega} |u_n|^{p^* \gamma} \, dx \right)^{\frac{1}{m'}}.
\end{align*}
By Lemma \ref{emb}, using the continuous embedding $HW_0^{1,p}(\Omega) \hookrightarrow L^{p^{*}}(\Omega)$ and the fact that $\frac{p}{p^{*}} - \frac{1}{m'} > 0$, we obtain
\[
\|u_n^\gamma\|_{L^{p^{*}}(\Omega)} \leq C,
\]
where $C$ is independent of $n$. Therefore, the sequence $\{u_n^\gamma\}_{n \in \mathbb{N}}$ is uniformly bounded in $L^t(\Omega)$ with $t = p^{*} \gamma$. Consequently, the pointwise limit $u$ belongs to $L^t(\Omega)$.

\item[$(ii)$] Let $k \geq 1$ and define $A(k) = \{ x \in \Omega : u_n(x) \geq k \}$. Choosing 
\[
\varphi_k(x) = (u_n - k)^+ \in HW_0^{1,p}(\Omega)
\] 
as a test function in \eqref{approxeqn}, and proceeding similarly to the proof of Theorem \ref{regthm}, the desired result follows.
\end{enumerate}

\textbf{Proof of Theorem \ref{regthm2}:}
\begin{enumerate}
\item[$(i)$] Observe that $m \in \left(1, \frac{p^{*}}{p^{*}-p}\right)$ implies 
\[
\gamma = \frac{p m'}{p m' - p^*} > 1.
\] 
Now, choosing 
\[
\varphi = u_n^{p\gamma - p + 1} \in HW_0^{1,p}(\Omega)
\] 
as a test function in \eqref{approxeqn}, and using Lemma \ref{emb} along with the continuous embedding 
\[
HW_0^{1,p}(\Omega) \hookrightarrow L^{p^*}(\Omega),
\] 
the result follows by proceeding similarly to the proof of Theorem \ref{regthm1}-$(i)$.
\item[$(ii)$] The result follows by arguments similar to those in the proof of Theorem \ref{regthm}.
\end{enumerate}

\textbf{Proof of Theorem \ref{regthm3}:}
\begin{enumerate}
\item[$(i)$] Observe that $m \in \left(1, \frac{p^{*}}{p^{*}-p}\right)$ implies 
\[
\gamma = \frac{(\delta+p-1)m'}{pm'-p^*} > \frac{\delta+p-1}{p} > 1,
\] 
since $\delta > 1$. Choosing now $\varphi = u_n^{p\gamma-p+1} \in HW_0^{1,p}(\Om)$ as a test function in \eqref{approxeqn}, and using Lemma \ref{emb} along with the continuous embedding $HW_0^{1,p}(\Om) \hookrightarrow L^{p^{*}}(\Omega)$, the result follows by arguing as in the proof of Theorem \ref{regthm1}-$(i)$.
\item[$(ii)$] The proof follows analogously to that of Theorem \ref{regthm}.
\end{enumerate}
\subsection{Proof of the uniqueness result}
\begin{proof}[Proof of Theorem \ref{thm4}]
If $u$ and $v$ are two weak solutions of \eqref{meqn} satisfying zero Dirichlet boundary conditions, then Theorem \ref{comparison} implies $u \leq v$. Similarly, one also obtains $v \leq u$.
\end{proof}

\section*{Acknowledgment}
This work is supported by ANRF research grant, file no: ANRF/ECRG/2024/000780/PMS.

\end{document}